\documentclass[11pt]{article}
\providecommand\IfDocumentMetadataT[1]{}

\usepackage{verbatim}

\usepackage[T1]{fontenc}
\usepackage{lmodern} 
\usepackage[a4paper,margin=1.1in]{geometry}

\usepackage{amsmath,amssymb,amsfonts,amsthm}
\usepackage{thmtools}
\usepackage{thm-restate}
\usepackage{graphicx}
\usepackage{float}
\usepackage{caption}
\usepackage{subcaption}
\usepackage{forest}
\usepackage{centernot}
\usepackage{xcolor}
\usepackage{microtype}
\usepackage{xurl}

\usepackage{tikz}
\tikzset{vertex/.style={circle,draw,inner sep=1.5pt}}
\usetikzlibrary{
  arrows.meta,
  decorations.pathreplacing
}

\usepackage[
  backend=biber,
  style=alphabetic,
  sorting=nty,
  giveninits=true,
  maxbibnames=99,
  doi=true,
  isbn=false,
  url=false,
  eprint=true
]{biblatex}
\renewbibmacro*{in:}{%
  \ifentrytype{article}
    {}
    {\printtext{\bibstring{in}\intitlepunct}}}
\DeclareNameAlias{author}{given-family}
\DeclareNameAlias{editor}{given-family}
\AtBeginBibliography{\emergencystretch=3em\color{black}\hypersetup{allcolors=black}}
\usepackage[
  colorlinks=true,
  linkcolor=blue,
  citecolor=blue,
  urlcolor=blue
]{hyperref}
\declaretheorem[
  name=Theorem,
  numberwithin=section
]{theorem}

\declaretheorem[
  sibling=theorem,
  name=Conjecture
]{conjecture}

\declaretheorem[
  sibling=theorem,
  name=Definition,
  style=definition
]{definition}

\declaretheorem[
  sibling=theorem,
  name=Question,
  style=definition
]{question}

\declaretheorem[
  sibling=theorem,
  name=Remark,
  style=remark
]{remark}

\title{Typical properties of countable graphs: flows and bridges}

\author{%
  Robert \v{S}\'amal\thanks{Computer Science Institute of Charles University,
  Prague, Czech Republic. E-mail: \texttt{samal@iuuk.mff.cuni.cz}.
  Supported by grant 25-16627S of the Czech Science Foundation.}
  \and
  Hadi Zamani\thanks{Computer Science Institute of Charles University, Prague,
  Czech Republic. E-mail: \texttt{hadiz@iuuk.mff.cuni.cz}.
  Supported by grant 25-16627S of the Czech Science Foundation, and grant
  SVV--2025--260822 of Charles University.}
}

\date{}

\begin{document}

\maketitle

\begin{abstract}
    This paper has two aims. First, we study nowhere-zero flows in countably infinite graphs,
    observing the different roles played by two types of bridge. Second, we ask how typical graphs with such bridges are. 
    Since there is no canonical probability measure on countable graphs, we explore an analytic approach
    and study which graph classes are nowhere dense and which are meager for two different 
    definitions of a distance. We characterize in this sense graphs with a bridge (separating a finite component, 
    or two infinite ones), $D$-edge-colorable graphs, etc.
\end{abstract}
\section{Introduction}

In the study of flows on finite graphs, bridges play a crucial role: every flow
assigns the value zero to each bridge. It is also well known that graphs with a
bridge are atypical: both dense random graphs and random regular graphs are bridgeless with high probability (whp)
\cite{erdHos1961strength,frieze2016introduction}. In this paper, we investigate
how these phenomena extend to countable graphs.

In Section~\ref{sec:flows}, we explain how the relationship between bridges and
flows differs in the infinite setting. Using compactness, we prove a version of
Seymour's 6-flow theorem for countable graphs.

In the rest of the paper, we investigate which properties of countable graphs are typical. 
We first explain why we explore an analytic (Baire category) notion of
typicality. Extending the Erdős--Rényi model to countably infinite
graphs almost surely yields the Rado graph, which follows by the standard back-and-forth argument
from the fact that the countable random graph almost surely satisfies the extension property~\cite{ER63}. 
In other words, up to isomorphism, the resulting probability measure is concentrated on a single graph.

One might instead try to extend the model of random $d$-regular graphs, for
instance, via the configuration model~\cite{bollobas1980probabilistic}. This
would require a uniform probability distribution on the set of perfect matchings
of a given countable set and, in particular, a uniform probability distribution on a
countable set itself (given by the distribution of neighbors of the first
vertex). Such a distribution does not exist, so this approach also fails. 

The above arguments concern permutation-invariant probability laws on a
fixed countable vertex set. Probability measures on rooted isomorphism
classes exist and are in fact standard in the study of local weak convergence and
unimodular random graphs~\cite{benjamini_schramm,aldous2007processes}. In particular, local weak limits of
finite graphs rooted at a uniformly chosen vertex give probability measures on
rooted graph spaces, and these measures satisfy the mass-transport principle. In
this paper, rather than choosing a particular law, we study a notion of
typicality determined by the local topology itself.

Consequently, we adopt the Baire category framework, familiar from real
analysis, as an alternative notion of typicality: a property is called
\emph{typical} if the class of graphs satisfying it is comeager (also called residual), and
\emph{atypical} if this class is meager. 

The use of Baire category in the study of infinite graphs and related graph limit objects is not new. 
The Rado graph also has a comeager isomorphism class in the space of graphs on a fixed countable vertex set; see 
\cite[Section 1.7]{cameron2013random}.
More substantially, Truss studied the typical automorphism of this Rado graph in this sense of residual subsets of $Aut(R)$ \cite{Truss}. 
Baire category methods are also used in descriptive combinatorics to study Baire-measurable colorings
and matchings on Borel graphs~\cite{conley2016brooks,marks2016baire}. In
graph limit theory, Lovász and Szegedy~\cite{lovasz2011finitely}
proved that finitely forcible graphons form a meager class in the $L^2$ topology. 
Barral Lijó and Nozawa~\cite{BLNchaos} showed genericity of their ``almost chaotic'' property
of rooted graphs. 

There is also a precise connection with probability. 
If \(X\) is a compact metric space and \(A\subseteq X\) is Borel, then \(A\) is meager if and only if
the set $\{\mu\in\mathcal P(X):\mu(A)=0\}$ is comeager in the space \(\mathcal P(X)\) of Borel probability measures equipped with the weak
topology~\cite[Theorem~3.16 and~3.17]{dubins1964measurable}. 
Applied to the compact rooted spaces of Section~\ref{sec:metric}, 
the category results also describe almost-sure behavior for a generic probability law, though not necessarily for any specified law.
On the other hand, this notion of typicality need not agree with the one known from familiar
random graph models. Uniform random finite \(D\)-regular graphs converge locally
to the infinite \(D\)-regular tree~\cite{vanderhofstad2017random}, whose edges are all two-way infinite bridges
when \(D\ge3\). In contrast, Corollary~\ref{rootedinfinitebridgelessset} shows that graphs containing a two-way
infinite bridge form a meager class in our rooted spaces.

We study meager and comeager sets defined by two different notions of closeness. 
In Section~\ref{sec:metric} we investigate typical properties of rooted graphs of bounded degree and achieve several 
positive, if surprising, results; e.g., we show that graphs with a bridge are typical
(in the maximum-degree space and in the odd-degree regular space), while graphs 
with a two-way infinite bridge (bridge separating two infinite components) are atypical. 
Our metric here is motivated by Benjamini--Schramm convergence~\cite{benjamini_schramm}: 
we consider graphs close if they are isomorphic in a large neighborhood of the root. 
To a graph theorists the meager property may sound unfamiliar; there is, however, a natural 
combinatorial interpretation in terms of infinite games (see Gr\"adel~\cite{Gradel} for details). 
Fix one of the rooted graph spaces considered in Section~\ref{sec:metric}. 
Two players agree on a graph property $\cal P$ and take turns building an
infinite graph. Each turn is an extension of some $r$-neighborhood of the root
to $r'$-neighborhood (for some $r'>r$) that has an infinite extension in the appropriate rooted space. 
After countably many moves, they will
have created an infinite graph; if it satisfies $\cal P$, the starting player wins, otherwise the second player wins. 
By the Banach--Mazur theorem (see \cite[Theorem 4]{Gradel}), the second player has a winning strategy 
if and only if $\cal P$ is meager. 

The properties we discuss are independent of the choice of the root. Still, the choice of root does 
affect the metric, so  in Section~\ref{sec:unrooted}, we turn to spaces of unrooted graphs (also with
bounded degree) and discover that the situation is more complicated. 
We use here the Hausdorff metric derived from the rooted metric and observe that this 
defines the naive convergence of Elek~\cite{elek2018qualitative}. 
We show that there are strongly isolated infinite
graphs, that is, graphs separated from every non-isomorphic graph by a
fixed positive distance. Any class containing such a graph is nonmeager.
For the bridge properties considered, we construct strongly isolated
examples both satisfying and violating the property, whenever permitted by
the relevant degree bounds. Consequently, in the corresponding unrooted
spaces, these properties are neither meager nor comeager, so the Baire
category framework does not classify them as typical or atypical.

\section{Nowhere-zero flows for infinite graphs}
\label{sec:flows}

\begin{definition}[Weak and strong flows]
Let $G$ be a locally finite graph with a fixed orientation $\vec{G}$, and let
$A$ be an abelian group. For a vertex $v$, let $\delta^+(v)$ and $\delta^-(v)$
denote the sets of edges directed away from and toward $v$, respectively. A
\emph{weak $A$-flow} is a function $\varphi\colon E(G)\to A$ satisfying
Kirchhoff's law at every vertex:
\[
  \sum_{e\in\delta^+(v)}\varphi(e) = \sum_{e\in\delta^-(v)}\varphi(e)
  \qquad\forall\, v\in V(G).
\]
For a set $S\subseteq V(G)$, let $\delta^+(S)$ (resp.\ $\delta^-(S)$) denote
the edges directed from $S$ to $V(G)\setminus S$ (resp.\ the reverse). A
\emph{strong $A$-flow} is a function $\varphi\colon E(G)\to A$ satisfying
Kirchhoff's law across every finite cut: whenever
$|\delta^+(S)|+|\delta^-(S)|<\infty$,
\[
  \sum_{e\in\delta^+(S)}\varphi(e) = \sum_{e\in\delta^-(S)}\varphi(e).
\]
Either flow is \emph{nowhere-zero} if $\varphi(e)\neq 0$ for all $e\in E(G)$.
For an integer $k\ge 2$, a (weak or strong) nowhere-zero \emph{$k$-flow} is a
(weak or strong) nowhere-zero $\mathbb{Z}_k$-flow.
\end{definition}

\begin{remark}
\label{rem:integerflows}
For finite graphs, a classical theorem of
Tutte~\cite{tutte1949imbedding} states that a nowhere-zero
$\mathbb{Z}_k$-flow exists if and only if a nowhere-zero integer-valued
flow with values in $\{\pm 1,\ldots,\pm(k-1)\}$ exists.
\end{remark}

\begin{definition}[Bridges in infinite graphs]
Let $G$ be an infinite connected graph. An edge $e\in E(G)$ is a
\emph{bridge} if $G-e$ is disconnected. A bridge is \emph{two-way infinite}
if both components of $G-e$ are infinite, and \emph{one-way infinite}
otherwise.
\end{definition}

Every strong flow is a weak flow. The two notions coincide for finite
graphs, but they diverge in the infinite setting. For example, a two-way
infinite path admits a weak nowhere-zero $A$-flow for every nontrivial
abelian group $A$, whereas it admits no strong nowhere-zero $A$-flow.

\begin{figure}
    \centering
\definecolor{figsolid}{HTML}{0072B2} 
\definecolor{figdashed}{HTML}{E69F00} 
\definecolor{figdotted}{HTML}{009E73} 

\begin{tikzpicture}[scale=.7,
  base edge/.style={line width=1.15pt,line cap=round,line join=round},
  solid edge/.style={base edge,draw=figsolid},
  dashed edge/.style={base edge,draw=figdashed,dashed},
  dotted edge/.style={base edge,draw=figdotted,densely dotted},
  vertex/.style={circle,fill=black,inner sep=2.8pt}
]

\def\n{4}
\def\step{4.2}

\def\yu{1.8}
\def\ym{0.9}
\def\yl{0.0}
\def\yt{3.0}
\def\yb{-1.25}

\def\r{1.05}
\pgfmathsetmacro{\cy}{\ym-\r*sin(18)}

\pgfmathsetmacro{\half}{\step/2}
\pgfmathsetmacro{\botL}{\half-0.9}
\pgfmathsetmacro{\botR}{\half+0.9}

\foreach \k in {0,1,2,3,4}{
  \pgfmathsetmacro{\x}{\k*\step}

  \coordinate (SU\k) at (\x,\yu);
  \coordinate (SL\k) at (\x,\yl);

  \draw[dotted edge] (SU\k)--(SL\k);
}

\foreach \k in {0,1,2,3}{
  \pgfmathsetmacro{\x}{\k*\step}
  \pgfmathsetmacro{\m}{\x+\half}

  \coordinate (T\k)  at (\m,\yt);
  \coordinate (BL\k) at (\x+\botL,\yb);
  \coordinate (BR\k) at (\x+\botR,\yb);

  \coordinate (C\k)  at (\m,\cy);
  \coordinate (P\k1) at ($(C\k)+(90:\r)$);
  \coordinate (P\k2) at ($(C\k)+(18:\r)$);
  \coordinate (P\k3) at ($(C\k)+(-54:\r)$);
  \coordinate (P\k4) at ($(C\k)+(-126:\r)$);
  \coordinate (P\k5) at ($(C\k)+(162:\r)$);

  \pgfmathtruncatemacro{\kp}{\k+1}
  \draw[dashed edge] (SU\k)--(T\k);
  \draw[solid edge] (T\k)--(SU\kp);
  \draw[solid edge] (SL\kp)--(BR\k);
  \draw[dotted edge] (BR\k)--(BL\k);
  \draw[dashed edge] (BL\k)--(SL\k);

  \draw[solid edge] (P\k1)--(P\k3);
  \draw[dotted edge] (P\k3)--(P\k5);
  \draw[solid edge] (P\k5)--(P\k2);
  \draw[dotted edge] (P\k2)--(P\k4);
  \draw[dashed edge] (P\k4)--(P\k1);

  \draw[dotted edge] (T\k)--(P\k1);
  \draw[dashed edge] (P\k3)--(BR\k);
  \draw[solid edge] (P\k4)--(BL\k);
}

\draw[dashed edge,<-] (-1.8,\ym)--(P05);
\draw[dashed edge] (P02)--(P15);
\draw[dashed edge] (P12)--(P25);
\draw[dashed edge] (P22)--(P35);
\draw[dashed edge,->] (P32)--(4*\step+1.8,\ym);

\draw[solid edge,<-] (-1.8,\yu)--(SU0);
\draw[solid edge,<-] (-1.8,\yl)--(SL0);

\draw[dashed edge,->] (SU4)--(4*\step+1.8,\yu);
\draw[dashed edge,->] (SL4)--(4*\step+1.8,\yl);

\foreach \k in {0,1,2,3,4}{
  \node[vertex] at (SU\k) {};
  \node[vertex] at (SL\k) {};
}

\foreach \k in {0,1,2,3}{
  \node[vertex] at (T\k) {};
  \node[vertex] at (BL\k) {};
  \node[vertex] at (BR\k) {};

  \foreach \i in {1,2,3,4,5}{
    \node[vertex] at (P\k\i) {};
  }
}

\end{tikzpicture}
\caption{A $3$-connected two-ended infinite graph $G$ with no strong nowhere-zero $4$-flow, but with a weak nowhere-zero $4$-flow (Miraftab and Moghadamzadeh~\cite{miraftaba2017algebraic}).
}
\label{fig:weakflow}
\end{figure}
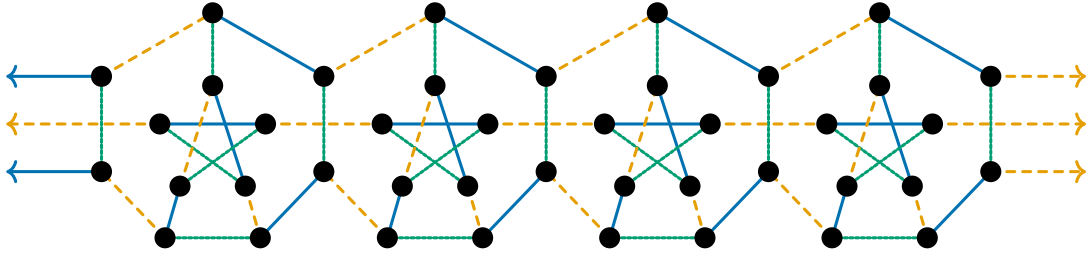

Miraftab and Moghadamzadeh~\cite{miraftaba2017algebraic} developed an
algebraic flow theory for infinite graphs which, in the terminology of the
present paper, corresponds to our strong-flow notion (called non-elusive flows in their paper). 
They constructed a $3$-connected two-ended infinite graph $G$ (Figure~\ref{fig:weakflow}) with
no strong nowhere-zero $4$-flow, although the same graph admits a weak
nowhere-zero $4$-flow in our sense. Using a compactness argument, they extended several classical results of finite flow theory to infinite graphs. In their more general framework, they proved that every bridgeless graph,
without assuming local finiteness, admits a strong nowhere-zero
$6$-flow~\cite[Theorem~25]{miraftaba2017algebraic}.

We prove a weak-flow counterpart to their result. By weakening the
hypothesis to forbid only one-way infinite bridges and applying logical
compactness, we show that such graphs always admit a weak nowhere-zero $6$-flow.

\begin{restatable}{theorem}{infbridgeless}
\label{infbridgeless}
A connected countably infinite, locally finite graph containing no one-way infinite
bridge admits a weak nowhere-zero $\mathbb{Z}_6$-flow.
\end{restatable}
The excluded bridge is also a necessary obstruction: summing the
vertex equations for vertices of its finite side forces its value to be zero.
 
\section{Metric spaces of rooted connected graphs}
\label{sec:metric}

To study which properties are typical, we follow an approach common in
analysis: we use a metric space of graphs and study nowhere dense and meager
sets, which we now define precisely.

The \emph{$\varepsilon$-neighborhood} of a point $x$ in a metric space
$(X,d)$ is defined as
\(
\mathcal{U}_{\varepsilon}(x) = \{ y \in X \mid d(x, y) < \varepsilon \}.
\)
Similarly, the \emph{closed $\varepsilon$-neighborhood}
$\overline{\mathcal{U}}_{\varepsilon}(x)$ is the set of all points
satisfying \( d(x, y) \le \varepsilon \).

A subset \( U \subseteq X \) of a metric space \( (X, d) \) is \emph{open}
if and only if for every point \( x \in U \) there exists
\( \varepsilon > 0 \) such that
\( \mathcal{U}_{\varepsilon}(x) \subseteq U \).

A subset \(S\subseteq X\) is \emph{nowhere dense} if and only if for every
non-empty open set \(U\subseteq X\) there exists a non-empty open set
\(U'\subseteq U\) such that \( U'\cap S=\emptyset \). Equivalently, for
every non-empty open neighborhood \(\mathcal{U}_{\varepsilon}(x)\) there
exists a non-empty open neighborhood
\(\mathcal{U}_{\delta}(y)\subseteq \mathcal{U}_{\varepsilon}(x)\) such that
\( \mathcal{U}_{\delta}(y)\cap S=\emptyset \).

A set is \emph{meager} if it can be written as a countable union of nowhere
dense sets. A set is \emph{comeager} if its complement is meager.

\begin{definition}[Spaces of rooted graphs and rooted digraphs]
Let $\mathfrak{G}^\bullet$ and $\vec{\mathfrak{G}}^\bullet$ denote the
spaces of connected rooted graphs and connected rooted digraphs,
respectively, considered up to root-preserving isomorphism. For $D \ge 3$,
let $\mathfrak{G}_{\le D}^\bullet$ denote the subspace of rooted graphs
with maximum degree at most $D$, and let $\mathfrak{G}_{=D}^\bullet$ denote
the subspace of rooted graphs in which every vertex has degree exactly $D$.
The superscript $\bullet,\infty$ indicates the corresponding subspace of
infinite rooted graphs. The analogous notation for rooted digraphs is
obtained by replacing $\mathfrak{G}$ with $\vec{\mathfrak{G}}$, where
degree means total degree.
\end{definition}

Throughout this section, we fix $D \ge 3$. Let $\mathfrak{X}$ be a rooted
graph space. For an integer $r \ge 0$, an \emph{$r$-ball} in $\mathfrak{X}$ is the rooted induced
subgraph of some $(G,x)\in \mathfrak{X}$ obtained by restricting $G$ to the
vertices at distance at most $r$ from the root $x$:
\[
B_r(G,x)=\bigl(G[\{ y \in V(G) \mid d_G(x,y)\le r\}],\,x\bigr).
\]

\begin{definition}[The metric space $(\mathfrak{G}^{\bullet}_{\le D},d^{\bullet})$]
For $(G_1,v_1),(G_2,v_2)\in\mathfrak{G}^{\bullet}_{\le D}$, define
\[
d^{\bullet}\bigl((G_1,v_1),(G_2,v_2)\bigr)
:=\inf\left\{2^{-r}
 \;\middle|\; B_{r}(G_1,v_1)\cong B_{r}(G_2,v_2)\right\}.
\]
\end{definition}

The space $(\mathfrak{G}^{\bullet}_{\le D}, d^{\bullet})$ is a metric space
\cite{lovasz2012large}; in particular, this is also true for
$(\mathfrak{G}^{\bullet,\infty}_{\le D}, d^{\bullet})$.

\begin{definition}[Extensions of $F$]
For an $r$-ball $F = B_r(G,v)$ with $(G,v) \in \mathfrak{G}^{\bullet}_{\le D}$,
define the set of \emph{extensions} of $F$ by
\[
\mathfrak{G}^{\bullet}_{F}
:=
\left\{
(G',v') \in \mathfrak{G}^{\bullet}_{\le D}
\;\middle|\;
B_r(G',v') \cong F
\right\}.
\]
Similarly, define
\[
\mathfrak{G}^{\bullet,\infty}_{F}
:=
\left\{
(G',v') \in \mathfrak{G}^{\bullet,\infty}_{\le D}
\;\middle|\;
B_r(G',v') \cong F
\right\}.
\]
\end{definition}

For an $r$-ball $F$, define \(L_r(F)\) to be the set of vertices of $F$ at
distance exactly $r$ from the root, and define
\(L_r^{<D}(F):=\{v\in L_r(F): \deg_F(v)<D\}\). Define
\(S(F):=\sum_{v\in L_r(F)}\bigl(D-\deg_F(v)\bigr)\).

\begin{restatable}{lemma}{basiclemma}
\label{basiclemma}
Let $F$ be an $r$-ball. Then:
\begin{itemize}
\item The set $\mathfrak{G}^{\bullet,\infty}_F$ is non-empty if and only if
$L^{<D}_r(F)$ contains at least one vertex.
\item The set
$\mathfrak{G}^{\bullet,\infty}_F \cap \mathfrak{G}^{\bullet}_{=D}$ is
non-empty if and only if $L^{<D}_r(F)$ contains at least one vertex and
every vertex of $V(F) \setminus L^{<D}_r(F)$ has degree $D$.
\end{itemize}
\end{restatable}

\begin{remark}
\label{rem:rootedbasics}
\begin{itemize}
\item The sets \(\mathfrak{G}^{\bullet}_{F}\) are both closed and open;
they form an open basis, and the space
\((\mathfrak{G}^{\bullet}_{\le D}, d^{\bullet})\) is compact and totally
disconnected~\cite{lovasz2012large}.
\item For each rooted graph \((G, x) \in \mathfrak{G}^{\bullet}_{\le D}\) and each natural number $r$,
we have
\[
\mathcal{U}_{2^{-(r-1)}}\bigl( (G, x)\bigr)
= \overline{\mathcal{U}}_{2^{-r}}\bigl( (G, x)\bigr)
= \mathfrak{G}^{\bullet}_{B_r(G, x)}.
\]
\item Let \(\mathcal{X}\subseteq \mathfrak{G}^{\bullet}_{\le D}\) be a
    subspace. A subset \(S\subseteq \mathcal{X}\) is nowhere dense in
    \(\mathcal{X}\) if and only if for every \(r\)-ball \(F\) with
    \((\mathfrak{G}^{\bullet}_{F}\cap \mathcal{X})\neq\emptyset\), there
    exists an \(r'\)-ball \(F'\) such that
    \[
    \emptyset\neq
    (\mathfrak{G}^{\bullet}_{F'}\cap \mathcal{X})
    \subseteq
    \bigl(\mathfrak{G}^{\bullet}_{F}\cap \mathcal{X}\bigr)\setminus S.
    \]
\item The spaces
    \(
    \mathfrak{G}^{\bullet}_{=D},
    \mathfrak{G}^{\bullet,\infty}_{\le D},
    \mathfrak{G}^{\bullet,\infty}_{=D}
    \)
    are closed subspaces of
    \((\mathfrak{G}^{\bullet}_{\le D}, d^{\bullet})\); for
    \(\mathfrak{G}^{\bullet,\infty}_{\le D}\) and
    \(\mathfrak{G}^{\bullet,\infty}_{=D}\) this follows from
    Lemma~\ref{basiclemma}. Hence they are compact and totally disconnected,
    and therefore, by the Baire Category Theorem, none of them is meager in
    itself.
\item The description of the extensions is the basis of the interpretation of our results in terms of Banach--Mazur games 
    (see the introduction). 
\item  Another use of the extension process is a definition of a probability
    measure on the rooted space: we can assign transition probabilities to
    ``one-step'' admissible extensions of finite rooted balls. A nested infinite
    sequence of balls determines a rooted graph. The assignment of a positive probability to
    every admissible one-radius extension gives a measure with full support.
    The resulting law depends on the transition probabilities and need not be
    unimodular. We do not pursue this approach further. 
    For the representation of rooted graphs by infinite branches in a tree of finite ball types, see~\cite{lehner2023note}. 
\end{itemize}
\end{remark}

\begin{restatable}{theorem}{rootedbridgelessset}
\label{rootedbridgelessset}
The set of graphs without a one-way infinite bridge is nowhere dense
\begin{itemize}
    \item in $\mathfrak{G}^{\bullet,\infty}_{\le D}$ for all $D\ge3$, and
    \item in $\mathfrak{G}^{\bullet,\infty}_{= D}$ for all odd $D\ge3$.
\end{itemize}
\end{restatable}

As a corollary, we get that in the bounded-degree rooted space~$\mathfrak{G}^{\bullet,\infty}_{\le D}$, 
and in the odd-degree regular rooted space~$\mathfrak{G}^{\bullet,\infty}_{=D}$, 
a generic graph admits no weak nowhere-zero flow over any abelian group.
On the other hand, for an even~$D$, assigning~1 to every edge gives us a weak nowhere-zero $\mathbb Z_2$-flow. 

The restriction to odd $D$ in the regular case is necessary. Indeed, when
$D$ is even and $e$ is a bridge in a $D$-regular graph $G$, each component
of $G-e$ has one boundary vertex of degree $D-1$, hence of odd degree,
while all its other vertices have degree $D$, hence even degree. The sum of
degrees in such a finite component is odd, so it cannot be finite. Hence one-way
infinite bridges cannot exist in $D$-regular graphs for even $D$. 
The natural analogue of Theorem~\ref{rootedbridgelessset} in this case involves
cut vertices instead. We call a non-root vertex~$v$ of an infinite graph a \emph{root-separating cut-vertex} if $G-v$~has the root in 
a finite component. 

\begin{restatable}{theorem}{cutvertex}
\label{cutvertex}
  The set of graphs without a root-separating cut-vertex is nowhere dense in
  $\mathfrak{G}^{\bullet,\infty}_{\le D}$ and
  $\mathfrak{G}^{\bullet,\infty}_{=D}$ for all $D\ge3$.
\end{restatable}

Note, that this in particular implies that a generic graph has \emph{some} cut vertex whose 
deletion leaves a finite component. This property is independent of the choice of the root 
(unlike the existence of the root-separating cut-vertex). 
Two other corollaries are as follows.

\begin{restatable}{corollary}{rootedoneended}
  \label{rootedoneended}
  The set of graphs with one end is comeager in
  $\mathfrak{G}^{\bullet,\infty}_{\le D}$ and
  $\mathfrak{G}^{\bullet,\infty}_{=D}$ for all $D\ge3$.
\end{restatable}

\begin{restatable}{corollary}{rootedinfinitebridgelessset}
\label{rootedinfinitebridgelessset}
  The set of graphs with a two-way infinite bridge is meager in
  $\mathfrak{G}^{\bullet,\infty}_{\le D}$ and
  $\mathfrak{G}^{\bullet,\infty}_{=D}$ for all $D\ge3$.
\end{restatable}

For $D$~even, we already argued that $D$-regular graphs have no one-way infinite bridges. 
The theorem above shows that a generic such graph has no bridge at all, thus 
it has a strong nowhere-zero $\mathbb Z_6$-flow by Theorem~25 of~\cite{miraftaba2017algebraic}. 
In fact, we can directly show existence of a strong nowhere-zero $\mathbb Z_2$-flow in such graph (generic $D$-regular 
for even~$D$): a constant $1$ is by definition a weak flow, and by Corollary~\ref{rootedoneended} 
all finite edge-cuts have one side finite; such cuts, however, must be even by the hand-shaking lemma. 

\begin{restatable}{theorem}{edgecoloring}
\label{edgecoloring}
The set of $D$-edge-colorable graphs is nowhere dense in
$\mathfrak{G}^{\bullet,\infty}_{\le D}$ and
$\mathfrak{G}^{\bullet,\infty}_{=D}$ for all $D\ge3$.
\end{restatable}

Together, Theorems~\ref{rootedbridgelessset}--\ref{edgecoloring} show that,
in the rooted setting, bridgelessness, the absence of cut vertices, the
presence of a two-way infinite bridge, and being $D$-edge-colorable, 
are all atypical. {The exception is the bridgelessness conclusion
in \(\mathfrak{G}_{=D}^{\bullet,\infty}\) when \(D\) is even, since an
even-regular graph cannot contain a one-way infinite bridge.}

\section{From rooted to unrooted graphs}
\label{sec:unrooted}

In the previous section we studied typicality in the space
$(\mathfrak{G}^{\bullet,\infty}_{\le D}, d^{\bullet})$ of infinite connected
rooted graphs. Let us illustrate the difference with the unrooted case with a simple example. 
Consider a  two-way infinite path $P^\bullet_\infty$ (with any root). 
Let $G^\bullet_n$ denote $P^\bullet_\infty$ with a leaf attached to a vertex at distance~$n$ from the root. 
The rooted distance of these graphs is 
$d^{\bullet}(P^\bullet_\infty, G^\bullet_n) = 2^{-n}$, so $P^\bullet_\infty$ is a limit point of graphs~$G^\bullet_n$. 
However, in many senses every graph~$G^\bullet_n$ is different from~$P^\bullet_\infty$ (e.g., existence of a one-way infinite bridge, 
thus also existence of a weak nowhere-zero flow). 
This leads us to the natural question: what happens if we drop the root? Can
we say something about the typicality of unrooted graphs? In this section, we
explore why it is, unfortunately, more complicated. 
A connected graph $G$ determines the set of all its
rootings $\{(G,v) : v \in V(G)\} \subseteq \mathfrak{G}^{\bullet}_{\le D}$,
and comparing these sets in the Hausdorff metric induced by $d^{\bullet}$
yields a natural pseudometric on the space of connected graphs.

\begin{definition}[Spaces of graphs and digraphs]
Let $\mathfrak{G}$ and $\vec{\mathfrak{G}}$ denote the spaces of connected
graphs and connected digraphs, respectively, considered up to isomorphism.
For $D \ge 3$, let $\mathfrak{G}_{\le D}$ denote the subspace of graphs with
maximum degree at most $D$, and let $\mathfrak{G}_{=D}$ denote the subspace
of graphs in which every vertex has degree exactly $D$. The superscript
$\infty$ indicates the corresponding subspace of infinite graphs. The
analogous notation for digraphs is obtained by replacing $\mathfrak{G}$ with
$\vec{\mathfrak{G}}$, where degree means total degree.
\end{definition}

Let $G$ be a graph with vertex set $V(G)$. For $r \in \mathbb{N}$, we define
the set of rooted $r$-balls of $G$ by
\[
\mathfrak{B}_r(G)
:=
\{\, B_r(G,v) \mid v \in V(G) \,\},
\]
where each ball is considered up to root-preserving isomorphism. We further
define
\[
\mathfrak{B}(G)
:=
\bigcup_{r \ge 0} \mathfrak{B}_r(G).
\]

\begin{definition}[The pseudometric \(d\) on \(\mathfrak{G}_{\le D}\)]
For \(G_1,G_2\in \mathfrak{G}_{\le D}\), define
\[
d(G_1,G_2)
:=
\inf\bigl\{2^{-r}\mid
\mathfrak{B}_r(G_1)\cong \mathfrak{B}_r(G_2)
\bigr\},
\]
where \(\mathfrak{B}_r(G_1)\cong \mathfrak{B}_r(G_2)\) means that there
exists a bijection between the two collections such that corresponding
rooted graphs are isomorphic via root-preserving isomorphisms.
\end{definition}

Equivalently, \(d\) is the Hausdorff pseudodistance, with respect to the rooted metric \(d^{\bullet}\), between the sets of rootings of the two graphs. More
explicitly, for all \(G_1,G_2\in\mathfrak{G}_{\le D}\),
\[
\begin{aligned}
d(G_1,G_2)
&=
d_H^{\bullet}\bigl(
\{(G_1,v):v\in V(G_1)\},
\{(G_2,w):w\in V(G_2)\}
\bigr) \\
&=
\max\Bigl\{
\sup_{v\in V(G_1)}\inf_{w\in V(G_2)}
d^{\bullet}\bigl((G_1,v),(G_2,w)\bigr), \\
&\hspace{1.3cm}
\sup_{w\in V(G_2)}\inf_{v\in V(G_1)}
d^{\bullet}\bigl((G_1,v),(G_2,w)\bigr)
\Bigr\}.
\end{aligned}
\]
Here rooted graphs are considered up to root-preserving isomorphism, so
multiplicities of isomorphic rootings are ignored.

If we restrict attention to finite graphs, then \(d\) is a metric. In
general, however, \((\mathfrak{G}_{\le D},d)\) is only a pseudometric space:
distinct non-isomorphic graphs may have distance zero; see
Theorem~\ref{nonisomorphictrees}.

\begin{remark}
The induced notion of convergence is the \emph{naive convergence} introduced by
Elek~\cite{elek2018qualitative}. Whereas Benjamini--Schramm convergence
measures the limiting frequencies of finite rooted neighborhoods, naive
convergence records only which finite rooted neighborhoods occur.
Elek considers countable bounded-degree graphs, not necessarily connected,
and identifies two graphs if they have the same collection of rooted balls.
Thus, he works with the metric quotient of the corresponding pseudometric
space. In contrast, we restrict to connected graphs and consider them up to
isomorphism without identifying graphs at pseudodistance zero. 
\end{remark}

\begin{definition}[Extensions of $\mathfrak{F}$]
Let $\mathfrak{F}$ be a collection of rooted $r$-balls of maximum degree at
most $D$, each considered up to root-preserving isomorphism. We define the
set of \emph{extensions} of $\mathfrak{F}$ by
\[
\mathfrak{G}_{\mathfrak{F}}
:=
\left\{
G' \in \mathfrak{G}_{\le D}
\;\middle|\;
\mathfrak{B}_r(G') \cong \mathfrak{F}
\right\}.
\]
\end{definition}

For each $G \in \mathfrak{G}_{\le D}$ and each $r \in \mathbb{N}$, we
have
\[
\mathcal{U}_{2^{-(r-1)}}(G)
= \overline{\mathcal{U}}_{2^{-r}}(G)
= \mathfrak{G}_{\mathfrak{B}_r(G)}.
\]
In particular, the sets $\mathfrak{G}_{\mathfrak{F}}$ are both closed and
open, and they form an open basis of $(\mathfrak{G}_{\le D}, d)$.

Let $\mathcal{X} \subseteq \mathfrak{G}_{\le D}$ be a subspace. A
subset $S \subseteq \mathcal{X}$ is nowhere dense in $\mathcal{X}$ if and
only if for every collection $\mathfrak{F}$ of rooted $r$-balls with
$(\mathfrak{G}_{\mathfrak{F}} \cap \mathcal{X}) \neq \emptyset$, there
exists a collection $\mathfrak{F}'$ of rooted $r'$-balls such that
\[
\emptyset
\neq (\mathfrak{G}_{\mathfrak{F}'} \cap \mathcal{X})
\subseteq (\mathfrak{G}_{\mathfrak{F}} \cap \mathcal{X}) \setminus S.
\]

\begin{definition}[Weakly isolated points]
Let $(\mathfrak{X},d)$ be a pseudometric space. A point
$x \in \mathfrak{X}$ is called \emph{weakly isolated} if there is no point
$x' \in \mathfrak{X}$ with $x' \not= x$ such that $d(x,x') = 0$.
\end{definition}

\begin{definition}[Strongly isolated points]
Let $(\mathfrak{X},d)$ be a pseudometric space. A point
$x \in \mathfrak{X}$ is called \emph{strongly isolated} if there is
$\varepsilon > 0$ such that no point $x' \in \mathfrak{X}$ with
$x' \not= x$ satisfies $d(x,x') \le \varepsilon$.
\end{definition}

\begin{remark}
When we speak of isolated graphs in \((\mathfrak{G}_{\le D},d)\), the
condition \(x' \neq x\) is understood up to isomorphism; that is, we
require \(G' \not\cong G\). Note that a graph
$G \in \mathfrak{G}_{\le D}$ is strongly isolated if and only if
$\mathfrak{G}_{\mathfrak{B}_r(G)} = \{G\}$ for some $r \in \mathbb{N}$; in
that case $\{G\}$ is open. 
\end{remark}

\begin{remark}
Strong isolation implies weak isolation, but the converse fails. Indeed,
the \(D\)-regular tree \(T_D\) is weakly isolated, since it is the unique
connected \(D\)-regular graph containing no cycle. However, it is not
strongly isolated, since there exists a sequence of finite \(D\)-regular
graphs whose girths tend to infinity and which converges to \(T_D\).
\end{remark}

\begin{remark}
Since \((\mathfrak{G}^{\bullet}_{\le D},d^{\bullet})\) is a metric space,
every rooted graph in \(\mathfrak{G}^{\bullet}_{\le D}\) is weakly
isolated.
\end{remark}

Theorem~\ref{nonisomorphictrees} shows that \(\mathfrak{G}_{\le D}\)
contains uncountably many infinite graphs that are not weakly isolated.
Moreover, Theorem~\ref{regularnonisomorphictrees} shows that
\(\mathfrak{G}_{=D}\) contains uncountably many infinite graphs that are
not weakly isolated.

\begin{restatable}{theorem}{nonisomorphictrees}
\label{nonisomorphictrees}
For every $D \ge 3$, there exist $2^{\aleph_0}$ pairwise non-isomorphic
infinite trees in $\mathfrak{G}_{\le D}^\infty$ with the same collection of
rooted balls $\mathfrak{B}$.
\end{restatable}

After finishing the proof of Theorem~\ref{nonisomorphictrees} we found that 
Elek~\cite{elek2018qualitative} gives a similar construction in his Proposition 6.2 (with a different proof), 
valid for $D > 3$. We give an explicit construction that also covers the case $D=3$. 
Note that it is impossible to extend this result to the $D$-regular graphs: there is only one countably infinite
$D$-regular tree (up to isomorphism). However, if we do not restrict ourselves to trees, we get the following:

\begin{restatable}{theorem}{regularnonisomorphic}
\label{regularnonisomorphictrees}
For every $D \ge 3$, there exist $2^{\aleph_0}$ pairwise non-isomorphic
graphs in $\mathfrak{G}_{=D}^\infty$ with the same collection of rooted
balls $\mathfrak{B}$.
\end{restatable}

Theorem~\ref{bridge} gives a further distinction: there exist an infinite regular graph with a two-way infinite bridge and an infinite regular bridgeless graph
whose distance is zero. Consequently, the set of graphs with a two-way infinite bridge is not Borel  in the unrooted naive topology.

\begin{restatable}{theorem}{bridge}
\label{bridge}
Determining whether a graph is bridgeless cannot be done by examining the
collection of rooted balls $\mathfrak{B}$ alone. More precisely, for every
$D \ge 3$, there exist $H_1^D, H_2^D \in \mathfrak{G}_{= D}^\infty$ such that
$d(H_1^D, H_2^D) = 0$, where one of the graphs contains a two-way infinite
bridge, while the other contains no bridge.
\end{restatable}

This is also a distinction between weak and strong flow theory:
the bridgeless graph in Theorem~\ref{bridge} admits a strong nowhere-zero
$6$-flow (Theorem~25 of \cite{miraftaba2017algebraic}), 
whereas the graph with a bridge does not. 
Thus the existence of a strong nowhere-zero $6$-flow is not
determined by $\mathfrak{B}(G)$. In contrast, the existence 
of one-way infinite bridges (and thus, by Theorem~\ref{infbridgeless}, the existence of a weak nowhere-zero $6$-flow)
is determined by $\mathfrak{B}(G)$.

Strongly isolated graphs play an important role in studying typicality due
to the following simple lemma.

\begin{restatable}{lemma}{isolateisnotmeager}
\label{lem:isolateisnotmeager}
A set that includes a strongly isolated point cannot be meager.
\end{restatable}

\begin{remark}
Finite rooted graphs, as well as finite unrooted graphs, are strongly
isolated: each is uniquely determined by neighborhoods of sufficiently
large radius. Hence, a class containing a finite (rooted) graph cannot be
meager. For this reason, our study of typicality is restricted to infinite
rooted graphs in Section~\ref{sec:metric} and to infinite graphs in
Section~\ref{sec:unrooted}.
\end{remark}

The following result shows that there are infinitely many strongly isolated
infinite graphs with bounded degrees.

\begin{restatable}{theorem}{stronglyisolated}
\label{stronglyisolated}
For $D\ge3$, there are countably infinitely many infinite graphs that are
strongly isolated in $\mathfrak{G}_{\le D}$, and hence also in
$\mathfrak{G}_{\le D}^{\infty}$.
\end{restatable}

The graphs we construct to prove Theorem~\ref{stronglyisolated} have
bounded degree but are not regular. The existence of strongly isolated
graphs in $\mathfrak{G}_{=D}^{\infty}$ is not explored in this paper; see
Question~\ref{q:regularisolated}.

\begin{remark}
Since we have restricted the space to infinite graphs, strong isolation in
$\mathfrak{G}_{=D}$ is not needed; being strongly isolated in
$\mathfrak{G}_{=D}^\infty$ suffices.
\end{remark}

\begin{restatable}{theorem}{notmeager}
\label{th:notmeager}
The following statements hold:
\begin{enumerate}
    \item[\textnormal{(i)}] For every $D \ge 3$, the set of infinite
    graphs with a one-way infinite bridge 
    is neither meager nor comeager in~$\mathfrak{G}_{\le D}^{\infty}$.
    \item[\textnormal{(ii)}] For every $D \ge 3$, the set of infinite
    graphs with a two-way infinite bridge 
    is neither meager nor comeager in~$\mathfrak{G}_{\le D}^{\infty}$.
    \item[\textnormal{(iii)}] For every $D \ge 4$, the set of infinite bridgeless graphs 
    is neither meager nor comeager in~$\mathfrak{G}_{\le D}^{\infty}$.
\end{enumerate}
\end{restatable}

\begin{restatable}{theorem}{unrootededgecoloring}
    \label{th:unrootededgecoloring}
    For every \(D\geq 3\), the set of $D$-edge-colorable graphs is neither meager nor
    comeager in~$\mathfrak{G}_{\le D}^{\infty}$.
\end{restatable}

\begin{remark}
The gadget construction below proves
Theorem~\ref{th:notmeager}\,\textnormal{(iii)} for $D\ge4$ and leaves the
case $D=3$ open. A bridgeless strongly isolated graph of maximum degree
$3$ would settle this remaining case, but strong isolation is only a
sufficient condition for nonmeagerness; a different category argument
could also resolve it.
\end{remark}

This contrasts with the rooted setting: by
Theorem~\ref{rootedbridgelessset} and Corollary~\ref{rootedinfinitebridgelessset},
the set of bridgeless rooted graphs and the set of rooted graphs with a
two-way infinite bridge are both meager and by Theorem~\ref{edgecoloring} the
set of $D$-edge colorable rooted graphs is nowhere dense.  However, the study
of typicality restricted to infinite regular graphs remains open.

Our proofs rely on finding strongly isolated graphs: a small enough ball
around such a graph contains no other graph. While for some graphs this 
property is easy to verify, and there are only countably many such graphs, 
deciding whether a given ball contains any graph at all is already undecidable.
Let us use $\mathfrak{G}^\infty_{\mathfrak{F}}$ for 
$\mathfrak{G}_{\mathfrak{F}} \cap \mathfrak{G}_{\le D}^{\infty}$. 
As a corollary of Theorem~1 in \cite{csoka2012undecidability} we get the following: 

\begin{theorem}[{Cs\'oka~\cite{csoka2012undecidability}}]
\label{csoka-infinite}
In $\mathfrak{G}_{\le D}^{\infty}$ for $D\ge 4$, given a finite collection $\mathfrak{F}$ of rooted $r$-balls, it is undecidable
whether $\mathfrak{G}_{\mathfrak{F}}^\infty$ is nonempty.
\end{theorem}

Based on this, we conjecture that testing whether a specified basic neighborhood is a singleton
is also undecidable. Thus, we expect there is no characterization of 
strongly isolated graphs. 

\begin{conjecture} 
\label{Undecidableuniquness}
Consider $\mathfrak{G}_{\le D}^{\infty}$ for $D\ge 3$. 
Assume $\mathfrak{F}$ is a finite collection of rooted $r$-balls such that
$\mathfrak{G}_{\mathfrak{F}}^\infty$ is nonempty. It is undecidable whether
$|\mathfrak{G}_{\mathfrak{F}}^\infty| > 1$.
\end{conjecture}

\subsection{The gadget construction}

The proofs of Theorems~\ref{stronglyisolated}, \ref{th:notmeager} and~\ref{th:unrootededgecoloring} rely
on a construction that produces strongly isolated graphs, which we now
describe. We first find strongly isolated points among directed graphs.

\begin{restatable}{lemma}{IsolatedDirected}
\label{IsolatedDirected}
For every \(1\le m\le D\), the infinite directed star \(\vec{S}_{m,\infty}\)
is strongly isolated in \(\vec{\mathfrak{G}}_{\le D}\), and hence also in
\(\vec{\mathfrak{G}}_{\le D}^{\infty}\). Here, \(\vec{S}_{m,\infty}\)
denotes the directed star with \(m\) infinite outgoing branches; see
Figure~\ref{fig:star} for the case \(m=3\).
\end{restatable}

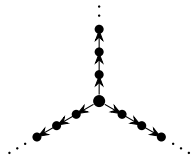
\begin{figure}[h]
\centering
\begin{tikzpicture}[>=Stealth, scale=.5]

\def\n{3}          
\def\R{0.7}        
\def\step{0.6}     
\def\shown{2}      
\def\dotrad{0.9pt} 

\node[circle, fill=black, inner sep=1.6pt] (c) at (0,0) {};

\foreach \i in {1,...,\n} {
    \pgfmathsetmacro{\angle}{90 + (\i-1)*360/\n}

    \coordinate (p\i0) at ({\R*cos(\angle)},{\R*sin(\angle)});
    \draw[->] (c) -- (p\i0);
    \node[circle, fill=black, inner sep=1.2pt] at (p\i0) {};

    \foreach \k in {1,...,\shown} {
        \pgfmathsetmacro{\distprev}{\R + (\k-1)*\step}
        \pgfmathsetmacro{\distnext}{\R + \k*\step}

        \coordinate (p\i\k prev) at ({\distprev*cos(\angle)},{\distprev*sin(\angle)});
        \coordinate (p\i\k)      at ({\distnext*cos(\angle)},{\distnext*sin(\angle)});

        \draw[->] (p\i\k prev) -- (p\i\k);
        \node[circle, fill=black, inner sep=1.2pt] at (p\i\k) {};
    }

    \pgfmathsetmacro{\dota}{\R + (\shown+0.6)*\step}
    \pgfmathsetmacro{\dotb}{\R + (\shown+1.0)*\step}
    \pgfmathsetmacro{\dotc}{\R + (\shown+1.4)*\step}

    \fill ({\dota*cos(\angle)},{\dota*sin(\angle)}) circle (\dotrad);
    \fill ({\dotb*cos(\angle)},{\dotb*sin(\angle)}) circle (\dotrad);
    \fill ({\dotc*cos(\angle)},{\dotc*sin(\angle)}) circle (\dotrad);
}

\end{tikzpicture}
\caption{Infinite directed 3-star $\vec{S}_{3,\infty}$.}
\label{fig:star}
\end{figure}

Now, our goal is to carry this over to the space of simple graphs. For this
purpose, we use the following construction, although other constructions
may also be considered.

Let $\vec{G}=(V, A)$ be an oriented graph. Let $H$ be a finite connected
simple graph with two distinguished vertices $v$ and $w$ such that no
automorphism $\varphi\in\operatorname{Aut}(H)$ satisfies $\varphi(v)=w$. We
regard $H$ as a two-terminal gadget with ordered terminals $(v,w)$, and
call it an \emph{asymmetric gadget}. 
We next describe a gadget operation that is a special case of 
a more general notion studied by Hartman, Hons, and Ne{\v{s}}et{\v{r}}il~\cite{hartman2025gadget}. 
We define an undirected graph $\vec{G}\ast H$
as follows. For each arc $(x,y)\in A$, take a vertex-disjoint copy
$H_{(x,y)}$ of $H$, identify its terminal $v$ with the vertex $x$, and
identify its terminal $w$ with the vertex $y$. Apart from these
identifications, all vertices and edges remain distinct.

Thus, each gadget $H$ induces a map from $\vec{\mathfrak{G}}_{\le D}$ to
$\mathfrak{G}$ sending $\vec{G}$ to $\vec{G}\ast H$. This map preserves
convergence. The following observation is an analogy of Theorem 3.1 
in Hartman, Hons, and Ne{\v{s}}et{\v{r}}il~\cite{hartman2025gadget}. Note, however, that 
convergence in Hausdorff pseudodistance (naive convergence of Elek) is \textbf{not} equivalent 
to $FO_0$-convergence of~\cite{hartman2025gadget}.
Anyway, our specific case is easy to show directly: for every $r$, the set of $r$-neighborhoods 
in $\vec{G} \ast H$ is determined by the set of $r$-neighborhoods in~$\vec{G}$.

\begin{restatable}{observation}{GadgetConvergenceLemma}
\label{GadgetConvergenceLemma}
Let $H$ be an asymmetric gadget. If $(\vec{G}_n)$ is a sequence of finite
oriented graphs such that $\vec{G}_n \longrightarrow \vec{G}$, then
$\vec{G}_n \ast H \longrightarrow \vec{G} \ast H$. In particular, if
$\vec{G}$ is approximable by finite directed graphs, then $\vec{G}\ast H$
is approximable by finite graphs.
\end{restatable}

The following question asks whether the converse implications hold. 

\begin{restatable}{question}{GadgetApproximabilityQ}
\label{GadgetApproximabilityQ}
Let $H$ be an asymmetric gadget, and let $\vec{G}$ be an infinite oriented
graph.
\begin{itemize}
    \item[\rm(i)] \textbf{Converse of
    Observation~\ref{GadgetConvergenceLemma}.} Let $(\vec{G}_n)$ be a
    sequence of finite oriented graphs. If
    $\vec{G}_n \ast H \longrightarrow \vec{G}\ast H$, does it follow that
    $\vec{G}_n \longrightarrow \vec{G}$?
    \item[\rm(ii)] \textbf{Converse of approximability.} If
    $\vec{G} \ast H$ is approximable by finite graphs, or by a sequence of
    pairwise non-isomorphic infinite graphs, does it follow respectively
    that $\vec{G}$ is approximable by finite oriented graphs, or by a
    sequence of pairwise non-isomorphic infinite oriented graphs?
\end{itemize}
\end{restatable}

In full generality, the answer to both parts of
Question~\ref{GadgetApproximabilityQ} is negative; counterexamples are
given in Figure~\ref{fig:negative-converse} for part~{\rm(i)} and
Figure~\ref{fig:negative-approximability} for part~{\rm(ii)}.

For part~{\rm(i)}, the asymmetric gadget $H$ of
Figure~\ref{fig:negative-converse} identifies two non-isomorphic oriented
graphs: $\vec{G}_1 \ncong \vec{G}_2$, but
$\vec{G}_1\ast H \cong \vec{G}_2\ast H$. Alternating accordingly between
two sequences yields a non-convergent sequence of oriented graphs whose
image under the gadget operation converges.

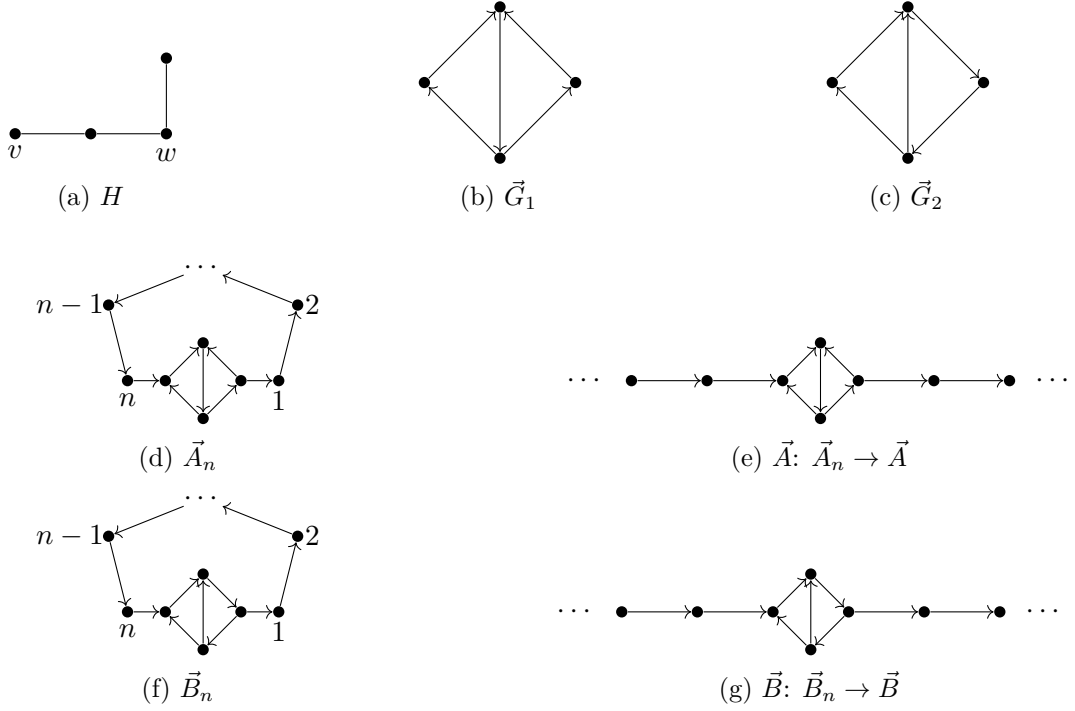
\begin{figure}[H]
        \centering

    \begin{subfigure}{0.3\textwidth}
        \centering
        \begin{tikzpicture}[
            scale=1,
            every node/.style={circle, fill=black, inner sep=1.5pt}
        ]
          \node (v1) at (0,0) {};
          \node (v2) at (1,0) {};
          \node (v3) at (2,0) {};
          \node (v4) at (2,1) {};

          \draw[-] (v1) -- (v2);
          \draw[-] (v2) -- (v3);
          \draw[-] (v3) -- (v4);

          \node[below=3pt, fill=none, inner sep=0pt] at (v1) {$v$};
          \node[below=3pt, fill=none, inner sep=0pt] at (v3) {$w$};
        \end{tikzpicture}
        \caption{$H$}
    \end{subfigure}
    \hfill
    \begin{subfigure}{0.3\textwidth}
        \centering
        \begin{tikzpicture}[
            scale=1,
            every node/.style={circle, fill=black, inner sep=1.5pt}
        ]
          \node (l) at (0,0) {};
          \node (u) at (1,1) {};
          \node (d) at (1,-1) {};
          \node (r) at (2,0) {};

          \draw[->] (l) -- (u);
          \draw[->] (d) -- (l);
          \draw[->] (d) -- (r);
          \draw[->] (r) -- (u);
          \draw[->] (u) -- (d);
        \end{tikzpicture}
        \caption{$\vec{G}_1$}
    \end{subfigure}
    \hfill
    \begin{subfigure}{0.3\textwidth}
        \centering
        \begin{tikzpicture}[
            scale=1,
            every node/.style={circle, fill=black, inner sep=1.5pt}
        ]
          \node (l) at (0,0) {};
          \node (u) at (1,1) {};
          \node (d) at (1,-1) {};
          \node (r) at (2,0) {};

          \draw[->] (l) -- (u);
          \draw[->] (u) -- (r);
          \draw[->] (r) -- (d);
          \draw[->] (d) -- (l);
          \draw[->] (d) -- (u);
        \end{tikzpicture}
        \caption{$\vec{G}_2$}
    \end{subfigure}

    \vspace{1em}

    \begin{subfigure}{0.45\textwidth}
        \centering
        \begin{tikzpicture}[xshift=-12,
            scale=.5,
            every node/.style={circle, fill=black, inner sep=1.5pt}
        ]
          \node (l) at (0,0) {};
          \node (u) at (1,1) {};
          \node (d) at (1,-1) {};
          \node (r) at (2,0) {};

          \draw[->] (l) -- (u);
          \draw[->] (d) -- (l);
          \draw[->] (d) -- (r);
          \draw[->] (r) -- (u);
          \draw[->] (u) -- (d);

          \node (p1) at (-1,0) {};
          \node (p2) at (-1.5,2) {};
          \node[draw=none, fill=none, inner sep=0pt] (p3) at (1,3) {$\cdots$};
          \node (p4) at (3.5,2) {};
          \node (p5) at (3,0) {};

          \node[below=3pt, fill=none, inner sep=0pt] at (p1) {$n$};
    
          \node[left=1pt, fill=none, inner sep=0pt] at (p2) {$n-1$};
           \node[right=1pt, fill=none, inner sep=0pt] at (p4) {$2$};
           
          \node[below=3pt, fill=none, inner sep=0pt] at (p5) {$1$};

          \draw[<-] (l) -- (p1);
          \draw[<-] (p1) -- (p2);
          \draw[<-] (p2) -- (p3);
          \draw[<-] (p3) -- (p4);
          \draw[<-] (p4) -- (p5);
          \draw[<-] (p5) -- (r);
        \end{tikzpicture}
        \caption{$\vec{A}_n$}
    \end{subfigure}
    \hfill
    \begin{subfigure}{0.45\textwidth}
        \centering
        \begin{tikzpicture}[
            scale=.5,
            every node/.style={circle, fill=black, inner sep=1.5pt}
        ]
          \node (l) at (0,0) {};
          \node (u) at (1,1) {};
          \node (d) at (1,-1) {};
          \node (r) at (2,0) {};

          \draw[->] (l) -- (u);
          \draw[->] (d) -- (l);
          \draw[->] (d) -- (r);
          \draw[->] (r) -- (u);
          \draw[->] (u) -- (d);

          \foreach \i in {1,2} {
            \pgfmathtruncatemacro{\x}{-2*\i}
            \node (L\i) at (\x,0) {};
            \ifnum\i=1
              \draw[->] (L\i) -- (l);
            \else
              \pgfmathtruncatemacro{\j}{\i-1}
              \draw[->] (L\i) -- (L\j);
            \fi
          }
          \node[fill=none] at (-5.2,0) {$\cdots$};

          \foreach \i in {1,2} {
            \pgfmathtruncatemacro{\x}{2*\i+2}
            \node (R\i) at (\x,0) {};
            \ifnum\i=1
              \draw[->] (r) -- (R\i);
            \else
              \pgfmathtruncatemacro{\j}{\i-1}
              \draw[->] (R\j) -- (R\i);
            \fi
          }
          \node[fill=none] at (7.2,0) {$\cdots$};
        \end{tikzpicture}
        \caption{$\vec{A}$: $\vec{A}_n\to\vec{A}$}
    \end{subfigure}

    \begin{subfigure}{0.45\textwidth}
        \centering
        \begin{tikzpicture}[
            scale=.5,
            every node/.style={circle, fill=black, inner sep=1.5pt}
        ]
          \node (l) at (0,0) {};
          \node (u) at (1,1) {};
          \node (d) at (1,-1) {};
          \node (r) at (2,0) {};

          \draw[->] (l) -- (u);
          \draw[->] (u) -- (r);
          \draw[->] (r) -- (d);
          \draw[->] (d) -- (l);
          \draw[->] (d) -- (u);
          
          \node (p1) at (-1,0) {};
          \node (p2) at (-1.5,2) {};
          \node[draw=none, fill=none, inner sep=0pt] (p3) at (1,3) {$\cdots$};
          \node (p4) at (3.5,2) {};
          \node (p5) at (3,0) {};

          \node[below=3pt, fill=none, inner sep=0pt] at (p1) {$n$};
    
          \node[left=1pt, fill=none, inner sep=0pt] at (p2) {$n-1$};
           \node[right=1pt, fill=none, inner sep=0pt] at (p4) {$2$};
           
          \node[below=3pt, fill=none, inner sep=0pt] at (p5) {$1$};

          \draw[<-] (l) -- (p1);
          \draw[<-] (p1) -- (p2);
          \draw[<-] (p2) -- (p3);
          \draw[<-] (p3) -- (p4);
          \draw[<-] (p4) -- (p5);
          \draw[<-] (p5) -- (r);
        \end{tikzpicture}
        \caption{$\vec{B}_n$}
    \end{subfigure}
\hfill
    \begin{subfigure}{0.45\textwidth}
        \centering
        \begin{tikzpicture}[
            scale=.5,
            every node/.style={circle, fill=black, inner sep=1.5pt}
        ]
          \node (l) at (0,0) {};
          \node (u) at (1,1) {};
          \node (d) at (1,-1) {};
          \node (r) at (2,0) {};

          \draw[->] (l) -- (u);
          \draw[->] (u) -- (r);
          \draw[->] (r) -- (d);
          \draw[->] (d) -- (l);
          \draw[->] (d) -- (u);

          \foreach \i in {1,2} {
            \pgfmathtruncatemacro{\x}{-2*\i}
            \node (L\i) at (\x,0) {};
            \ifnum\i=1
              \draw[->] (L\i) -- (l);
            \else
              \pgfmathtruncatemacro{\j}{\i-1}
              \draw[->] (L\i) -- (L\j);
            \fi
          }
          \node[fill=none] at (-5.2,0) {$\cdots$};

          \foreach \i in {1,2} {
            \pgfmathtruncatemacro{\x}{2*\i+2}
            \node (R\i) at (\x,0) {};
            \ifnum\i=1
              \draw[->] (r) -- (R\i);
            \else
              \pgfmathtruncatemacro{\j}{\i-1}
              \draw[->] (R\j) -- (R\i);
            \fi
          }
          \node[fill=none] at (7.2,0) {$\cdots$};
          
        \end{tikzpicture}
        \caption{$\vec{B}$: $\vec{B}_n\to\vec{B}$}

    \end{subfigure}
    \caption{Negative answer to
    Question~\ref{GadgetApproximabilityQ}{\rm(i)}. We have
    $\vec{G}_1 \ncong \vec{G}_2$, but
    $\vec{G}_1\ast H \cong \vec{G}_2\ast H$. Moreover, if
    $\vec{C}_n=\vec{A}_n$ for odd $n$ and $\vec{C}_n=\vec{B}_n$ for even
    $n$, then $(\vec{C}_n)$ is not convergent, while $(\vec{C}_n\ast H)$
    converges to $\vec{A}\ast H$.}
\label{fig:negative-converse}
\end{figure}

For part~{\rm(ii)}, Figure~\ref{fig:negative-approximability} shows an
asymmetric gadget $H$ and an infinite oriented graph $\vec{G}$ such that
$\vec{G}\ast H$ is approximable by finite graphs and also by a sequence of
pairwise non-isomorphic infinite graphs, while $\vec{G}$ is not
approximable by finite oriented graphs, nor by a sequence of pairwise
non-isomorphic infinite oriented graphs.

\begin{figure}[H]
\input{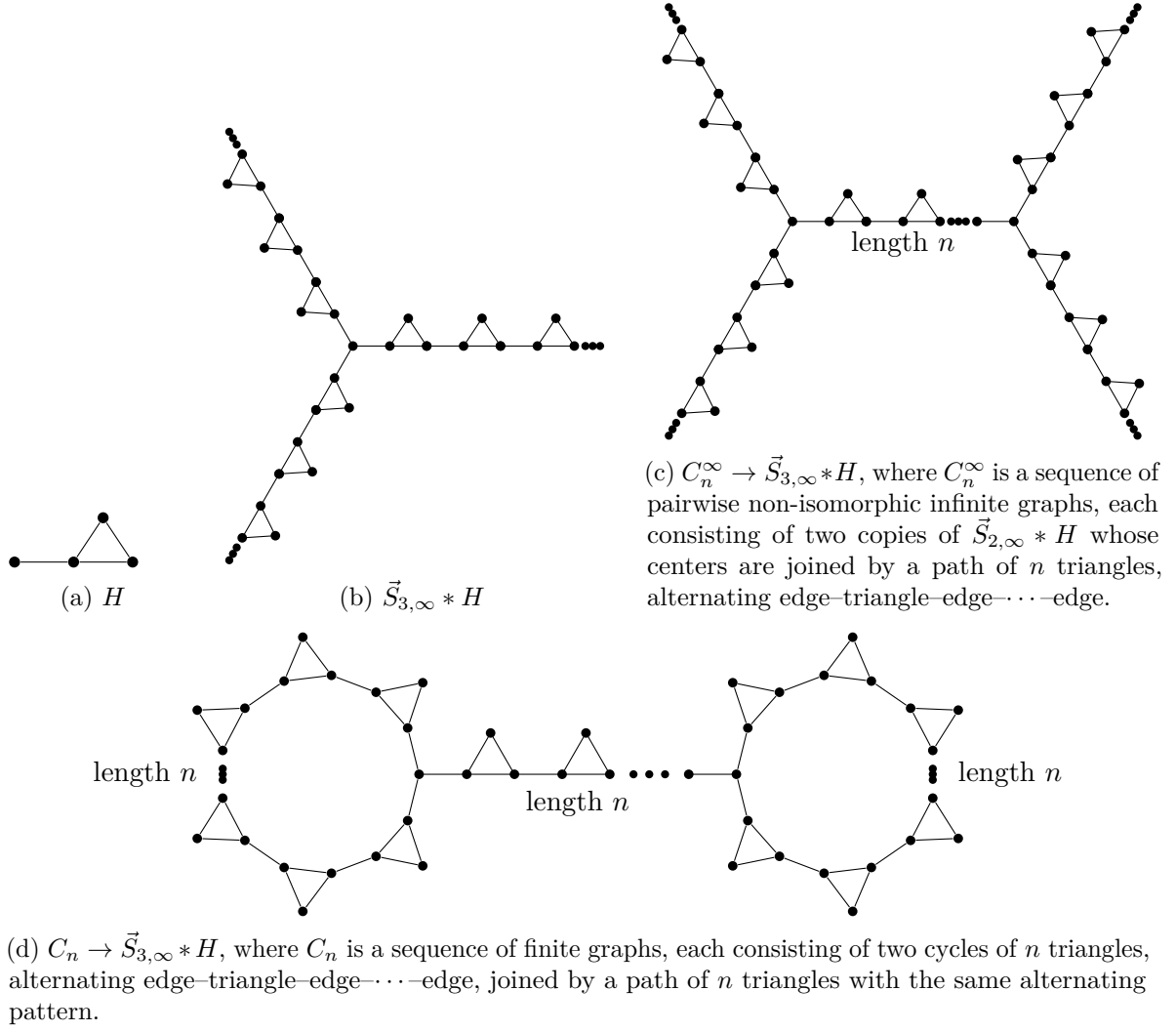}
\caption{Negative answer to
Question~\ref{GadgetApproximabilityQ}{\rm(ii)}.}
\label{fig:negative-approximability}
\end{figure}

However, there are asymmetric gadgets $H$ for which both converse
statements hold.

\begin{definition}[Unique decomposition property]
An asymmetric gadget $H$ has the \emph{unique decomposition property} if,
for every oriented graph $\vec{G}$, the induced subgraphs of
$\vec{G}\ast H$ that are isomorphic to $H$ decompose the entire edge set of
$\vec{G}\ast H$; that is, every edge of $\vec{G}\ast H$ belongs to exactly
one such copy of $H$.
\end{definition}

\begin{figure}[H]
\begin{subfigure}{0.45\textwidth}
\centering
\begin{tikzpicture}[scale=.9, every node/.style={circle, fill=black, inner sep=1.5pt}]
  \foreach \i in {1,...,8} {
    \node (v\i) at (\i-1,0) {};
  }
  \draw (v1)--(v2)--(v3)--(v4)--(v5)--(v6)--(v7)--(v8);
  \node (a) at (1,1) {};
  \node (b) at (2,1) {};
  \node (A) at (5,.8) {};
  \node (B) at (5.5,1) {};
  \node (O) at (4.5,.8) {};

  \node (C) at (6,.8) {};
  \draw (v2)--(a)--(b)--(v3);
  \draw (C)--(B)--(A);
  \draw (A)--(O);
  \draw (C)--(A);

  \draw (C)--(v7);

  \foreach \i in {2,...,7} {
    \node[below=3pt, fill=none, inner sep=0pt] at (v\i) {};
  }
  \node[above=3pt, fill=none, inner sep=0pt] at (a) {$p_1$};
  \node[above=3pt, fill=none, inner sep=0pt] at (b) {$p_2$};
  \node[above=3pt, fill=none, inner sep=0pt] at (A) {$q_1$};
  \node[above=3pt, fill=none, inner sep=0pt] at (C) {$q_2$};
  \node[below=3pt, fill=none, inner sep=0pt] at (v1) {$v$};
  \node[below=3pt, fill=none, inner sep=0pt] at (v8) {$w$};
\end{tikzpicture}

\caption{The gadget $H_0$}
\label{fig:gadget3}
\end{subfigure}
\begin{subfigure}{0.45\textwidth}
\centering
\begin{tikzpicture}[scale=.9, every node/.style={circle, fill=black, inner sep=1.5pt}]
  \foreach \i in {1,...,8} {
    \node (v\i) at (\i-1,0) {};
  }
  \draw (v1)--(v2)--(v3)--(v4)--(v5)--(v6)--(v7)--(v8);
  \node (a) at (1,1) {};
  \node (b) at (2,1) {};
  \node (A) at (5,.8) {};
  \node (B) at (5.5,1) {};
  \node (C) at (6,.8) {};
  \node[above=3pt, fill=none, inner sep=0pt] at (a) {$p_1$};
  \node[above=3pt, fill=none, inner sep=0pt] at (b) {$p_2$};
  \node[above=3pt, fill=none, inner sep=0pt] at (A) {$q_1$};

  \node[above=3pt, fill=none, inner sep=0pt] at (C) {$q_2$};
  \node[below=3pt, fill=none, inner sep=0pt] at (v1) {$v$};
  \node[below=3pt, fill=none, inner sep=0pt] at (v8) {$w$};

  \draw (v2)--(a)--(b)--(v3);
   \draw (v6)--(A)--(B)--(C)--(v7);

\end{tikzpicture}

\caption{The gadget $H_1$}
\label{fig:gadget1}
\end{subfigure}
\centering
\begin{subfigure}{0.45\textwidth}
\begin{tikzpicture}[scale=.9, every node/.style={circle, fill=black, inner sep=1.5pt}]

  \foreach \i in {2,...,7} {
    \node (u\i) at (\i-1,-1) {};
    }
  \foreach \i in {2,...,7} {
    \node (v\i) at (\i-1,0) {};
    }
    
  \node[label=below:$v$] (v1) at (0,-.5){};
  \node[label=below:$w$] (v8) at (7,-.5){};
  \node[above=3pt, fill=none, inner sep=0pt] at (a) {$p_1$};
  \node[above=3pt, fill=none, inner sep=0pt] at (b) {$p_2$};
  \node[above=3pt, fill=none, inner sep=0pt] at (A) {$q_1$};

  \node[above=3pt, fill=none, inner sep=0pt] at (C) {$q_2$};
  
  \draw (v1)--(v2)--(v3)--(v4)--(v5)--(v6)--(v7)--(v8);
  \draw (v1)--(u2)--(u3)--(u4)--(u5)--(u6)--(u7)--(v8);
  \node (a) at (1,1) {};
  \node (b) at (2,1) {};
  \node (A) at (5,.8) {};
  \node (B) at (5.5,1) {};
  \node (C) at (6,.8) {};

  \draw (v2)--(a)--(b)--(v3);
   \draw (v6)--(A)--(B)--(C)--(v7);

  \foreach \i in {1,...,8} {
    \node[below=3pt, fill=none, inner sep=0pt] at (v\i) {};
  }

\end{tikzpicture}
\caption{The gadget $H_2$}
\label{fig:gadget2}
\end{subfigure}
\caption{$H_0, H_1, H_2$ have the unique decomposition property.}
\label{fig:gadget}
\end{figure}
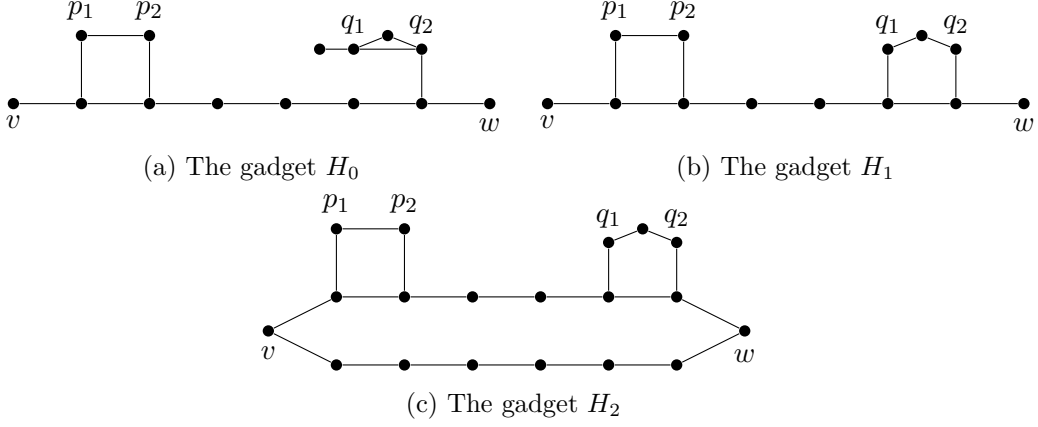

\begin{restatable}{lemma}{Uniquedecompositionlemma}
\label{lem:unique-decomposition}
Let $H_0$, $H_1$, $H_2$ be the gadgets depicted in
Figure~\ref{fig:gadget}.
\begin{itemize}
    \item Each of the gadgets $H_0$, $H_1$, $H_2$ has the unique
    decomposition property and only identity as an automorphism. 
    Moreover, for each $i \in \{0,1,2\}$ and each $k \ge 0$,
    the iterated subdivision $\operatorname{sd}^{k}(H_i)$ has the same two
    properties.
    \item For every $m \ge 3$, the graphs
    $\vec{S}_{m,\infty} \ast \operatorname{sd}^{k}(H_0)$ and
    $\vec{S}_{m,\infty} \ast \operatorname{sd}^{k}(H_1)$, $k \ge 0$, are
    pairwise non-isomorphic and each has maximum degree $m$.
    \item For every $m \ge 2$, the graphs
    $\vec{S}_{m,\infty} \ast \operatorname{sd}^{k}(H_2)$, $k \ge 0$, are
    pairwise non-isomorphic and each has maximum degree $2m$.
\end{itemize}
\end{restatable}

\begin{restatable}{lemma}{GadgetDecodingLemma}
\label{GadgetDecodingLemma}
Let $\vec{G}$ be an oriented (infinite) graph, and let $H$ be an asymmetric
gadget with the unique decomposition property such that every automorphism
of $H$ fixes both terminals. If $X_n \to \vec{G}\ast H$, then there exist
$N\in \mathbb{N}$ and directed graphs $\vec{G}_n$, for all $n\ge N$, such
that
\[
X_n = \vec{G}_n\ast H
\qquad\text{and}\qquad
\vec{G}_n \to \vec{G}.
\]
\end{restatable}

\begin{restatable}{corollary}{isolatedproduct}
\label{cor:isolatedproduct}
Let $1 \le m \le D$, and let $H$ be an asymmetric gadget with the unique
decomposition property such that every automorphism of $H$ fixes both
terminals. If $\vec{S}_{m,\infty} \ast H \in \mathfrak{G}_{\le D}$, then
$\vec{S}_{m,\infty} \ast H$ is strongly isolated in $\mathfrak{G}_{\le D}$.
\end{restatable}

\section{Proofs}

\subsection{Proofs for Section~\ref{sec:flows}}

\infbridgeless*
\begin{proof}
Let \(G\) be a connected countably infinite locally finite graph containing
no one-way infinite bridge. Write
\(
V(G)=\{v_1,v_2,\ldots\},
\)
orient \(\{v_i,v_j\}\) from \(v_i\) to \(v_j\) whenever \(i<j\), and take
\(v_1\) as the root. For an edge \(e\) incident with \(v\), define
\[
\operatorname{orient}(v,e):=
\begin{cases}
 1,  & \text{if \(e\) is directed towards \(v\)},\\
-1,  & \text{if \(e\) is directed away from \(v\)}.
\end{cases}
\]

For each \(r\in\mathbb N\), construct a finite multigraph \(G_r\) on
\(V(B_r(G,v_1))\cup\{v_\infty\}\) as follows. Start with \(B_r(G,v_1)\), and for every edge
\(\{v_i,v_j\}\in E(G)\) satisfying
$d(v_1,v_i)=r$ and $d(v_1,v_j)=r+1$, 
add an edge \(\{v_i,v_\infty\}\). Multiple edges are retained.

Equivalently, \(G_r\) is obtained by contracting
\(V(G)\setminus V(B_r(G,v_1))\) to \(v_\infty\) and deleting the resulting
loops. The graph \(G_r\) is bridgeless. Indeed, suppose that \(e\) were a
bridge of \(G_r\), and let \(C\) be the component of \(G_r-e\) not containing
\(v_\infty\). Then \(C\) is finite, and the edge of \(G\) corresponding to
\(e\) is the unique edge leaving \(C\). Hence it is a one-way infinite bridge
of \(G\), a contradiction.

Every edge of \(G_r\) corresponds naturally to an edge of \(G\); we label it
by that edge. Set
\begin{equation}
\label{eq:flow-values}
A_6:=\{1,2,3,4,5\},
\qquad
F_{(e,c)}
\quad(e\in E(G),\ c\in A_6),
\end{equation}
where \(F_{(e,c)}\) asserts that \(e\) receives the value \(c\).
For every edge \(e\), impose
\begin{equation}
\label{eq:edge-constraint}
F_e:=
\bigvee_{c\in A_6}
\left(
F_{(e,c)}
\land
\bigwedge_{d\in A_6\setminus\{c\}}\neg F_{(e,d)}
\right).
\end{equation}
Thus \(F_e\) requires \(e\) to receive exactly one nonzero value.

Let \(v\neq v_\infty\), and let \(e_1,\ldots,e_{\deg(v)}\) be the edges
incident with \(v\), labeled by their corresponding edges of \(G\). Define
\begin{equation}
\label{eq:vertex-constraint}
\operatorname{Sat}(v):=
\left\{
(c_1,\ldots,c_{\deg(v)})\in A_6^{\deg(v)}:
\sum_{j=1}^{\deg(v)}
\operatorname{orient}(v,e_j)c_j
\equiv 0\pmod 6
\right\},
\end{equation}
and put
\[
F_v:=
\bigvee_{(c_1,\ldots,c_{\deg(v)})\in\operatorname{Sat}(v)}
\left(
\bigwedge_{j=1}^{\deg(v)}F_{(e_j,c_j)}
\right).
\]
The formula \(F_v\) enforces Kirchhoff's law modulo \(6\) at \(v\).

Now set
\[
F_{G_r}:=
\left(
\bigwedge_{v\in V(G_r)\setminus\{v_\infty\}}F_v
\right)
\land
\left(
\bigwedge_{e\in E(G_r)}F_e
\right).
\]
Kirchhoff's law at all vertices other than \(v_\infty\) automatically implies
it at \(v_\infty\), since summing all vertex equations counts every edge once
with each sign. Consequently, \(F_{G_r}\) is satisfiable if and only if
\(G_r\) admits a nowhere-zero \(\mathbb Z_6\)-flow.

Since \(G_r\) is finite and bridgeless, Seymour's \(6\)-flow
theorem~\cite{seymour1981nowhere} gives an integer \(6\)-flow on \(G_r\).
Reducing its values modulo \(6\) shows that \(F_{G_r}\) is satisfiable.

For \(G\), consider the set of Boolean constraints
\[
\Sigma_G:=
\{F_e:e\in E(G)\}
\cup
\{F_v:v\in V(G)\}.
\]
Let \(\Phi\subseteq\Sigma_G\) be finite. Since \(G\) is locally finite, only
finitely many edges occur in the formulas belonging to \(\Phi\). Choose
\(r\) large enough that all their endpoints lie in \(B_r(G,v_1)\). Then every
constraint in \(\Phi\) occurs in \(F_{G_r}\), and hence \(\Phi\) is
satisfiable.

Thus every finite subset of \(\Sigma_G\) is satisfiable. By the compactness
theorem for Boolean formulas~\cite{godel1930vollstandigkeit},
\(\Sigma_G\) is satisfiable. The resulting edge values form a weak
nowhere-zero \(\mathbb Z_6\)-flow on \(G\).
\end{proof}

\begin{remark}
The integer-valued version follows from the same proof. Replace \(A_6\) in
\eqref{eq:flow-values} by
\[
A_6^{\mathbb Z}:=\{\pm1,\ldots,\pm5\}
\]
and replace the congruence in \eqref{eq:vertex-constraint} by equality in
\(\mathbb Z\). Seymour's theorem supplies the required integer \(6\)-flows
on the finite graphs \(G_r\), so the same compactness argument applies.
\end{remark}

\subsection{Proofs for Section~\ref{sec:metric}}

\basiclemma*
\begin{proof}
One implication is immediate: every rooted \(r\)-neighborhood of an infinite
graph in either space satisfies the corresponding condition stated above.

For the other implication, suppose that \(F\) is an \(r\)-ball satisfying the
stated condition. There is at least one vertex in \(L_r^{<D}(F)\). To each such vertex \(u\), we attach
\(D-\deg(u)\) disjoint infinite trees by adding one edge from \(u\) to the root
of each tree, so that in total $S(F)$ disjoint trees are attached. Each of these trees is chosen so that all its vertices have degree
\(D\), except for its root, which has degree \(D-1\) before the edge to \(u\) is
added. In this way, we obtain an infinite graph in
\(\mathfrak{G}_{\le D}^{\bullet,\infty}\), respectively in
\(\mathfrak{G}_{=D}^{\bullet,\infty}\), whose \(r\)-ball is equal to \(F\).
Thus \(\mathfrak{G}_{F}^{\bullet,\infty}\) is non-empty in both cases.
\end{proof}
\rootedbridgelessset*
\begin{proof}
To prove that the set of graphs without a one-way infinite bridge is nowhere
dense, we must show that for each non-empty open set
$\mathfrak{G}^{\bullet,\infty}_{F}$, there exists a non-empty open set
$\mathfrak{G}^{\bullet,\infty}_{F^\prime}$ such that every graph in
$\mathfrak{G}^{\bullet,\infty}_{F^\prime}$ contains a bridge. 

\begin{itemize}
    \item[$\mathfrak{G}^{\bullet,\infty}_{\le D}$:]
    For $\mathfrak{G}^{\bullet,\infty}_{\le D}$, we construct an $(r+1)$-ball $F^\prime$ simply as follows: we add one vertex and connect it to one vertex in $L_r^{<D}(F)$ (Figure~\ref{fig:image0}). This constructs an $(r+1)$-ball $F^\prime$ that contains a one-way infinite bridge.
    \item[$\mathfrak{G}^{\bullet,\infty}_{= D}$:]
    For $\mathfrak{G}^{\bullet,\infty}_{= D}$, we must also handle regularity:

\begin{figure}[H]
    \centering
         \centering
     \begin{subfigure}[b]{0.23\textwidth}
         \centering
         \begin{tikzpicture}[scale=.7]
            \draw (-1,0) -- (0,2) -- (1,0) -- cycle;
            \draw[dashed] (-1.5,-1) -- (0,2) -- (1.5,-1) -- (2,-2) -- (1,-2)  -- (+1,-1) --  (-1,-1) -- (-1,-2)  --(-2,-2)-- cycle;

            \draw (0,-1) -- (0,-2)  node[midway, xshift=-0.25cm,yshift=-0.1cm, rotate=90, style={scale=.8}] {$bridge$};
            \node[draw,circle] (u) at (0,-2.3) {};
            \node[draw,circle, fill=black ] (v) at (0,2) {};
            \node[] (v2) at (-1.1,2) {$root$};
            \draw [decorate,decoration={brace,amplitude=5pt,mirror,raise=4ex}]
              (-1.1,2) -- (-1.1,0) node[midway,xshift=-3em]{$F$};
            \draw [decorate,decoration={brace,amplitude=5pt,mirror,raise=4ex}]
              (-2,2) -- (-2,-2) node[midway,xshift=-3em]{$F^\prime$};
        \end{tikzpicture}
        \caption{$F'$ which contains a bridge}
        \label{fig:image0}
     \end{subfigure}
          \hfill
    \begin{subfigure}[b]{0.23\textwidth}
         \centering
         \begin{tikzpicture}[scale=1, every node/.style={draw, circle, inner sep=1.5pt}]
            \node[style={scale=.7}] (v2) at (0,0) {$K_{D+1}-e$};
            \node[fill=red!70] (u1) at (+.2,1) {};
            \node[fill=red!70] (u2) at (-.2,1) {};
            \draw (u1) -- (v2);
            \draw (u2) -- (v2);
        \end{tikzpicture}
        \caption{A gadget to reduce $S(F)$ by 2.}
        \label{fig:image2}
     \end{subfigure}
    \hfill
    \begin{subfigure}[b]{0.23\textwidth}
         \centering
         \begin{tikzpicture}[scale=.5, every node/.style={draw, circle, inner sep=1.5pt}]
            \def\n{6}
            \foreach \number in {1,...,\n}{
                \node[vertex, style={scale=.6}] (N-\number) at ({\number*(360/\n)}:2cm) {$w_\number$};
            }
            \node[fill=red!70] (u1) at ({1*(360/\n)}:3.5cm) {};
            \node[fill=red!70] (u2) at ({2*(360/\n)}:3.5cm) {};
            \draw (u1) -- (N-1);
            \draw (u2) -- (N-2);
            \foreach \number in {1,...,\n}{
                \foreach \y in {3,...,\n}{
                    \draw (N-\number) -- (N-\y);
                }
            }
        \end{tikzpicture}
        \caption{An example for $D=5$, for (b).}
        \label{fig:image3}
     \end{subfigure}
    \hfill
     \begin{subfigure}[b]{0.23\textwidth}
         \centering
         \begin{tikzpicture}[scale=1, every node/.style={draw, circle, inner sep=1.5pt}]
            \node[style={scale=.7}] (v2) at (0,0) {$K_{D+1}-e$};

            \node[fill=red!70] (u1) at (0,1) {};
            \node[label=below:{$u'$}] (u2) at (0,-1) {};
            \draw (u1) -- (v2);
            \draw (u2) -- (v2);
        \end{tikzpicture}
        \caption{A gadget for creating a bridge in $F'$.}
        \label{fig:image1}
     \end{subfigure}
     \hfill
    \begin{subfigure}[b]{0.23\textwidth}
         \centering
         \begin{tikzpicture}[scale=1, every node/.style={inner sep=1.5pt}]
            \node[draw,circle
            ,style={scale=.8}] (v1) at (0,0) {$w$};

            \node[draw,circle, fill=red!70] (u1) at (0,1) {};

            \node[draw,circle]  (u2) at (-1.5,-1) {};
            \node[draw,circle] (u3) at (-.5,-1) {};
            \node (u4) at (.5,-1) {$\dots$};
            \node[draw,circle] (u5) at (1.5,-1) {};

            \node at (-1.5,-1.45) {$u'_1$};
            \node at (-.5,-1.45) {$u'_2$};
            \node at (1.5,-1.45) {$u'_{D-1}$};

            \draw (u1) -- (v1);
            \draw (u2) -- (v1);
            \draw (u3) -- (v1);
            \draw (u5) -- (v1);
        \end{tikzpicture}
        \caption{A treelike expansion to increase the size of $L_r$.}
        \label{fig:image4}
     \end{subfigure}
     \hfill
     \begin{subfigure}[b]{0.23\textwidth}
         \centering
         \begin{tikzpicture}[scale=1, every node/.style={inner sep=1.5pt}]
            \node[draw,circle] (v1) at (0,0) {$w$};
            \node[draw,circle, fill=red!70] (u2) at (-1.5,1) {};
            \node[draw,circle, fill=red!70] (u3) at (-.5,1) {};
            \node (u4) at (.5,1) {};
            \node[draw,circle, fill=red!70]  (u5) at (1.5,1) {};
            \node (v2) at (-1.5,1.3) {$1$};
            \node (v3) at (-.5,1.3) {$2$};
            \node (u4) at (.5,1.3) {$\dots$};
            \node (v5) at (1.5,1.3) {${D}$};
            \draw (u2) -- (v1);
            \draw (u3) -- (v1);
            \draw (u5) -- (v1);
        \end{tikzpicture}
        \caption{A gadget to change the parity of $S(F)$.}
        \label{fig:image5}
     \end{subfigure}
     \hfill
     \begin{subfigure}[b]{0.23\textwidth}
         \centering
         \begin{tikzpicture}[scale=1, every node/.style={inner sep=1.5pt}]
            \node[draw,circle] (v1) at (0,0) {$w$};

            \node[draw,circle] (u2) at (-1.5,-.8) {};
            \node[draw,circle] (u3) at (-.5,-.8) {};
            \node (u4) at (.5,-.8) {$\dots$};
            \node[draw,circle] (u5) at (1.5,-.8) {};

            \node at (-1.5,-1.25) {$u'_1$};
            \node at (-.5,-1.25) {$u'_2$};
            \node at (1.5,-1.25) {$u'_{D-2}$};

            \node[draw,circle, fill=red!70] (g1) at (-.5,+.8) {};
            \node[draw,circle, fill=red!70] (g2) at (+.5,+.8) {};

            \draw (u2) -- (v1);
            \draw (u3) -- (v1);
            \draw (u5) -- (v1);
            \draw (g1) -- (v1);
            \draw (g2) -- (v1);
        \end{tikzpicture}
        \caption{A gadget for creating a vertex cut.}
        \label{fig:image6}
     \end{subfigure}
    
    
    \caption{Gadgets used in the construction of $F'$. Vertices labeled $w$ already have degree $D$ and cannot be extended. The terminals of the gadgets are colored red.}
\end{figure}
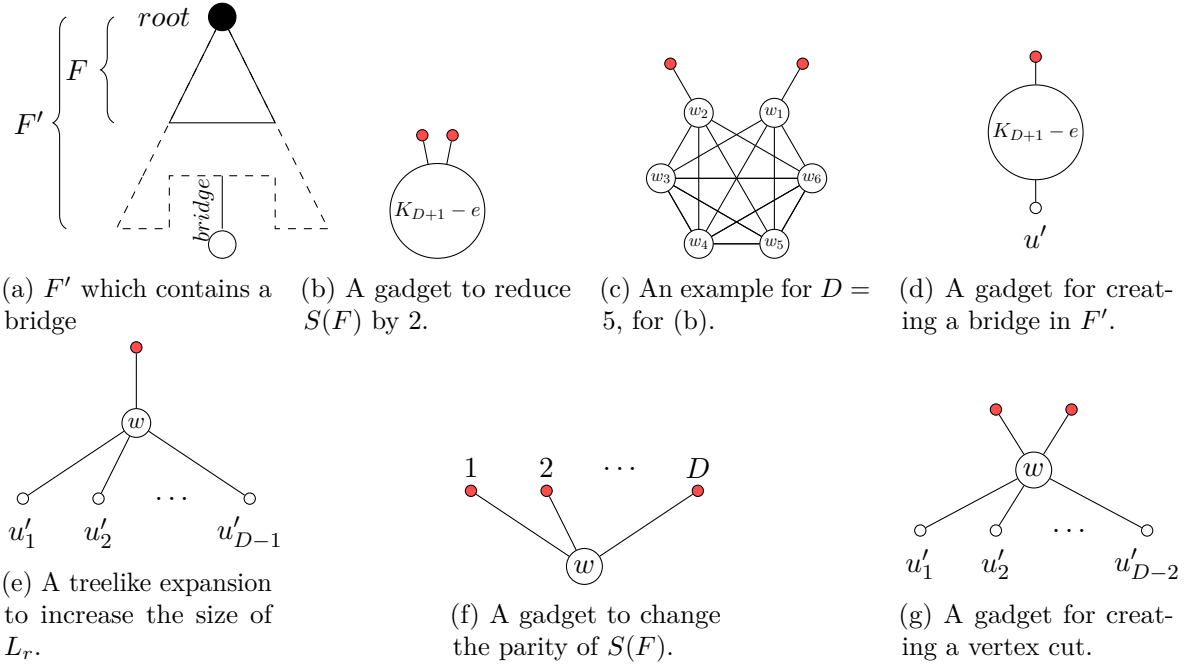

\begin{itemize}
\item[1.] If $S(F)$ is odd, we add $K_{D+1}-e$ gadgets (Figure~\ref{fig:image2}) to vertices in $L^{<D}_r(F)$ until $S(F) = 1$. We then add another type of $K_{D+1}-e$ gadget, connecting one of its vertices to $L_r^{<D}(F)$ and the other to a new vertex (Figure~\ref{fig:image1}), where the distance between the new vertex and the old vertex is $4$, and all the other vertices except the new one have degree $D$. This process constructs an $(r+4)$-ball $F^\prime$ such that each of its infinite extensions contains a one-way infinite bridge.

\item[2.] If $S(F)$ is even, and $L^{<D}_r(F)$ contains at least $D$ vertices, we then add a new vertex connected to $D$ of these vertices (Figure~\ref{fig:image5}). At this point, $S(F)$ becomes odd, allowing us to proceed as in the first case. This process constructs an $(r+4)$-ball $F^\prime$ 
such that each of its infinite extensions contains a one-way infinite bridge.

\item[3.] If $S(F)$ is even, and $L^{<D}_r(F)$ contains fewer than $D$ vertices, then we add $S(F)$ tree-like gadgets (Figure~\ref{fig:image4}), resulting in an $(r+2)$-ball $\tilde{F}$. For this new graph, $S(\tilde{F}) = (D-1)^2 \times S(F)$, which is even, and $L^{<D}_{\tilde{r}}(\tilde{F})$ contains more than $D$ vertices. Hence, we can proceed similarly to the second case. This process constructs an $(r+6)$-ball $F^\prime$ 
such that each of its infinite extensions contains a one-way infinite bridge.
\end{itemize}
    
\end{itemize}

In both spaces $\mathfrak{G}^{\bullet,\infty}_{F^\prime}$ is non-empty by the basic lemma (Lemma~\ref{basiclemma}), and every graph in $\mathfrak{G}^{\bullet,\infty}_{F^\prime}$ contains the same one-way infinite bridge.

\end{proof}

\cutvertex*
\begin{proof}
    For $\mathfrak{G}^{\bullet,\infty}_{\le D}$ with $D > 2$ and
    $\mathfrak{G}^{\bullet,\infty}_{= D}$ with odd $D > 2$, the result follows
    directly from Theorem~\ref{rootedbridgelessset}.  
    
    For $\mathfrak{G}^{\bullet,\infty}_{= D}$ with even $D > 2$, we proceed
    similarly to the previous proof. We show that for each
    $\mathfrak{G}^{\bullet,\infty}_{F}$, there exists a non-empty open set
    $\mathfrak{G}^{\bullet,\infty}_{F^\prime}$ such that every graph in
    $\mathfrak{G}^{\bullet,\infty}_{F^\prime}$ contains a cut vertex.  
    
    By parity, $S(F)$ is always even. If $L_r^{<D}(F)$ contains only one vertex,
    that vertex is a cut vertex. If $L_r^{<D}(F)$ contains more than one vertex, we
    add a new vertex and connect it to two vertices in $L_r^{<D}(F)$ and connect it
    to $D-2$ new vertices, creating a gadget as shown in Figure~\ref{fig:image6}.
    Since $S(F)$ remains even, we then add enough  $K_{D+1}-e$ gadgets (as in
    Figure~\ref{fig:image2}) to vertices in $L_r^{<D}(F)$ to make all the vertices
    in $F$ have degree $D$. This process constructs an $(r+2)$-ball $F^\prime$ that
    contains a cut vertex, and it is non-empty by the basic lemma
    (Lemma~\ref{basiclemma}).
\end{proof}

\rootedoneended*
\begin{proof}
    For any integer $r \ge 0$, let $X_r$ be the set of graphs with a finite set~$S$ in the $r$-ball around the root,
    such that removing $S$ creates more than one infinite component. We will show that $X_r$ is nowhere-dense, 
    thus showing that their union is meager. To this end, take an open ball 
    $\mathfrak{G}^{\bullet}_{F}$, we may assume $F$ is an $R$-ball around the root for some $R > r$. 
    By Theorem~\ref{cutvertex} there is an extension $F' \supset F$ such that 
    $\mathfrak{G}^{\bullet}_{F'}$ has no graph without a root-separating cut-vertex; the proof of Theorem~\ref{cutvertex}
    implies that the cut-vertex~$c$ satisfies $d_G(root,c) \ge R > r$. 
    It follows removing any subset of~$B_r(G,root)$ can create at most one infinite component, which we wanted to prove. 
\end{proof}

\rootedinfinitebridgelessset*
\begin{proof}
    The set of graphs with a two-way infinite bridge is a countable union of
    sets of graphs that have a two-way infinite bridge at distance $k$ from the
    root. Therefore, it suffices to prove that having a two-way infinite bridge
    at distance $k$ is nowhere dense.  
    
    Fix $k$. Given an $r$-ball $F$, we construct an $r^\prime$-ball $F^\prime$
    such that no two-way infinite bridge exists at level $k$. 
    To do this, notice that the proof of Theorem~\ref{cutvertex} gives us a ball $F' \supset F$ such that 
    every graph~$G$ in~$\mathfrak{G}^{\bullet,\infty}_{F^\prime}$ has a cut vertex~$c$ at level $> k$ from the root
    such that root is in a finite component of~$G-c$. 
    No such graph can have a two-way infinite bridge at distance $k$ from the root: indeed, removal of the bridge
    leaves us with two components and only one can contain~$c$; the other one is necessarily finite. 
\end{proof}

\edgecoloring*
\begin{proof}
To prove that the set of graphs is nowhere dense, again following the
definition, we must show that for each non-empty open set
\(\mathfrak{G}^{\bullet,\infty}_{F}\), there exists a non-empty open set
\(\mathfrak{G}^{\bullet,\infty}_{F'}\) such that every graph in
\(\mathfrak{G}^{\bullet,\infty}_{F'}\) has edge chromatic number \(D+1\).

Since the edge chromatic number of a graph is at least that of each of its
subgraphs, we proceed as follows. We attach one copy of the gadget \(J_D\)
described below to the \(r\)-ball \(F\). The gadget \(J_D\) is
\(D\)-regular at all its internal vertices, is not \(D\)-edge-colorable, and
has two deficient vertices of degree \(D-1\). We connect one of these
vertices to a vertex in \(L_r^{<D}(F)\) and the other to a new vertex,
similarly to what is shown in Figure~\ref{fig:image1} for even \(D\). We
complete the remaining deficiencies as in Lemma~\ref{basiclemma}. This
constructs an \(r'\)-ball \(F'\).

Every graph in \(\mathfrak{G}^{\bullet,\infty}_{F'}\), considered either in
\(\mathfrak{G}^{\bullet,\infty}_{=D}\) or in
\(\mathfrak{G}^{\bullet,\infty}_{\leq D}\), has edge chromatic number
\(D+1\). Indeed, it is not \(D\)-edge-colorable since it contains \(J_D\) as
a subgraph. On the other hand, every finite subgraph is
\((D+1)\)-edge-colorable by Vizing's theorem. The upper bound then follows
from~\cite{bruijn1951colour}: a graph is \(c\)-colorable if and only if each
of its finite subgraphs is \(c\)-colorable. The same statement applies to
edge colorability by passing to the line graph.

We now describe the construction of \(J_D\) and explain why it has the
required property. For every \(D\geq 3\), there exists a finite
\(D\)-regular graph \(G\) with
\(
\chi'(G)=D+1.
\)
For even \(D\), take \(G=K_{D+1}\): each color class in a
\(D\)-edge-coloring would be a perfect matching, which is impossible because
\(G\) has an odd number of vertices.

For odd \(D\), we construct a finite \(D\)-regular graph with a bridge.
Start with \(K_{D+2}\), choose a vertex \(x\) and two other vertices \(a,b\),
and delete the edges \(xa\) and \(xb\), together with a perfect matching on
the remaining \(D-1\) vertices. The resulting graph \(H\) has degree \(D-1\)
at \(x\) and degree \(D\) at every other vertex. Take two disjoint copies of
\(H\) and join their exceptional vertices by an edge \(e\). The resulting
graph \(G\) is \(D\)-regular, and \(e\) is a bridge.

Such a graph cannot be \(D\)-edge-colorable. Indeed, in a
\(D\)-edge-coloring every color class is a perfect matching. Each component
of \(G-e\) has odd order. Therefore, every color class not containing \(e\)
would give a perfect matching inside each component of \(G-e\), which is
impossible. Thus, by Vizing's theorem, \(\chi'(G)=D+1\).

We now remove a non-bridge edge \(uv\) from \(G\) and set
\(
J_D:=G-uv.
\)
Such an edge exists; for example, in the construction for odd \(D\), the
edge \(ab\) lies in a triangle and is therefore not a bridge. Choosing a
non-bridge edge ensures that \(J_D\) remains connected. Moreover,
\(
\chi'(J_D)=D+1.
\)
Suppose, for a contradiction, that \(J_D\) is \(D\)-edge-colorable. The
vertices \(u\) and \(v\) have degree \(D-1\), whereas every other vertex has
degree \(D\). Hence, precisely one color is missing at \(u\) and precisely
one color is missing at \(v\).

If \(|V(G)|\) is odd, then every color class misses an odd, and hence at
least one, number of vertices. Thus, there are at least \(D>2\) missing
incidences, but there are only two missing incidences in total, a
contradiction.

If \(|V(G)|\) is even, then every color class misses an even number of
vertices. Since \(u\) and \(v\) are the only vertices at which a color is
missing, the same color must be missing at both \(u\) and \(v\). Coloring
\(uv\) with this color yields a \(D\)-edge-coloring of \(G\), again a
contradiction. Therefore, by Vizing's theorem, \(\chi'(J_D)=D+1\).

Thus, for even \(D\), the gadget \(J_D\) is \(K_{D+1}\) with any edge
removed; for odd \(D\), it is the \(D\)-regular graph with a bridge
constructed above with a non-bridge edge removed. In both cases \(J_D\) is
connected, has exactly two vertices of degree \(D-1\), and satisfies
\(\chi'(J_D)=D+1\). Attaching one new edge to each of these two vertices
makes them \(D\)-regular while preserving \(J_D\) as a subgraph.
\end{proof}

\subsection{Proofs for Section~\ref{sec:unrooted}}
The proofs in this subsection are arranged according to logical dependence
rather than theorem numbering. We prove Theorem~\ref{bridge} before
Theorem~\ref{regularnonisomorphictrees}, since the latter uses the
construction from the former. We then prove the results concerning the
gadget construction before applying them to
Theorems~\ref{stronglyisolated}, \ref{th:notmeager}~and~\ref{th:unrootededgecoloring}. 

\nonisomorphictrees*

\begin{proof}
We construct a class of infinite trees using a class of infinite binary sequences.
Let
\[
\Omega=\{\omega=(\omega_n)_{n\in \mathbb{N}}\in\{0,1\}^{\mathbb{N}}:
\omega_1=1 \text{ and } \omega \text{ contains infinitely many }1\text{'s}\}.
\]
There are $2^{\aleph_0}$ such sequences.

Given such a sequence \(\omega\in\Omega\), we construct an infinite tree
\(T_\omega\) (Figure~\ref{fig:infinitetree}) as follows. Let
\(
\mathfrak{T}:=\{T_\omega:\omega\in\Omega\}
\)
be the corresponding class of infinite trees.

Let \((T_i)_{i\ge 1}\) be a fixed enumeration of the finite rooted trees with
maximum degree at most \(D\); see Figure~\ref{fig:treeslist}. The set of these
trees is denoted by
\(
\mathcal{T}^D:=\{T_i:i\ge 1\}.
\)
\begin{figure}[H]   
        \centering

    \begin{subfigure}{0.95\textwidth}
        \centering
        \begin{tikzpicture}[every node/.style={circle,draw,fill=black,inner sep=1pt}, level distance=1cm]

            \node[label=above:$T_1$, inner sep=3pt] (T1root) at (0,0) {}
                child {node {}};

            \node[label=above:$T_2$, inner sep=3pt] (T2root) at (2.5,0) {}
                child {node {}}
                child {node {}};

            \node[label=above:$T_3$, inner sep=3pt] (T3root) at (5.5,0) {}
                child {node {}}
                child {node {}}
                child {node {}};

            \node[label=above:$T_4$, inner sep=3pt] (T4root) at (9,0) {}
                child {node {}
                    child {node {}}
                };

            \node[draw=none,fill=none] at (12,0) {$\cdots$};

        \end{tikzpicture}
        \caption{The beginning of a fixed ordering of the finite rooted trees in \(\mathfrak{G}^{\bullet}_{\leq 3}\).}
        \label{fig:treeslist}
    \end{subfigure}

    \medskip

    \begin{subfigure}{0.95\textwidth}
        \centering
        \begin{tikzpicture}[thick,scale=.9]
            \node[circle, draw, minimum size=0.8cm] at (0,0) (N1) {$1$};

            \node[circle, draw, minimum size=0.8cm] at (2,0) (E1) {$2$};
            \node[circle, draw, minimum size=0.8cm] at (4,0) (E2) {3};
            \node[circle, draw, minimum size=0.8cm] at (6,0) (E3) {$4$};

            \draw (N1) -- (E1);
            \draw (E1) -- (E2);
            \draw (E2) -- (E3);
            \draw (E3) -- +(1.5,0);
            \node at (8,0) {$\cdots$};

            \node[circle, draw, minimum size=0.8cm] at (-2,0) (O1) {};
            \node[circle, draw, minimum size=0.8cm] at (-4,0) (O2) {};
            \node[circle, draw, minimum size=0.8cm] at (-6,0) (O3) {};

            \draw (O1) -- (N1);
            \draw (O2) -- (O1);
            \draw (O3) -- (O2);
            \draw (O3) -- +(-1.5,0);
            \node at (-8,0) {$\cdots$};
        \end{tikzpicture}
        \caption{The indexing of the vertices on the two-way infinite path.}
        \label{fig:pathindex}
    \end{subfigure}

    \medskip

    \begin{subfigure}{0.95\textwidth}
        \centering
        \begin{tikzpicture}[thick,scale=.9]
            \node[circle, draw, minimum size=0.8cm] at (0,0) (o3) {3};

            \node[circle, draw, minimum size=0.8cm] at (2,0) (o4) {4};
            \node[circle, draw, minimum size=0.8cm] at (4,0) (o5) {5};
            \node[circle, draw, minimum size=0.8cm] at (6,0) (o6) {6};

            \node[circle, draw, minimum size=0.8cm] at (-2,0) (o2) {2};
            \node[circle, draw, minimum size=0.8cm] at (-4,0) (o1) {1};
            \node[circle, draw, minimum size=0.8cm] at (-6,0) (o0) {};

            \draw (o0) -- (o1) -- (o2) -- (o3)-- (o4)--(o5)-- (o6);
            \draw (o6) -- +(1.5,0);
            \node at (8,0) {$\cdots$};
            \draw (o0) -- +(-1.5,0);
            \node at (-8,0) {$\cdots$};

            \node[regular polygon, regular polygon sides=3, draw, minimum size=0.8cm] at (-4,1.2) (T1) {$T_1$};
            \node[regular polygon, regular polygon sides=3, draw, minimum size=0.8cm] at (-2,1.2) (T2) {$T_2$};
            \node[regular polygon, regular polygon sides=3, draw, minimum size=0.8cm] at (2,1.2) (T3) {$T_3$};

            \draw (o1) -- (T1);
            \draw (o2) -- (T2);
            \draw (o4) -- (T3);
        \end{tikzpicture}
        \caption{An example of the construction of \(T_\omega\) for \(\omega=1101001\cdots\).}
        \label{fig:infinitetree}
    \end{subfigure}
    \caption{The construction of \(T_\omega\).}
    \label{fig:tree-construction}
\end{figure}
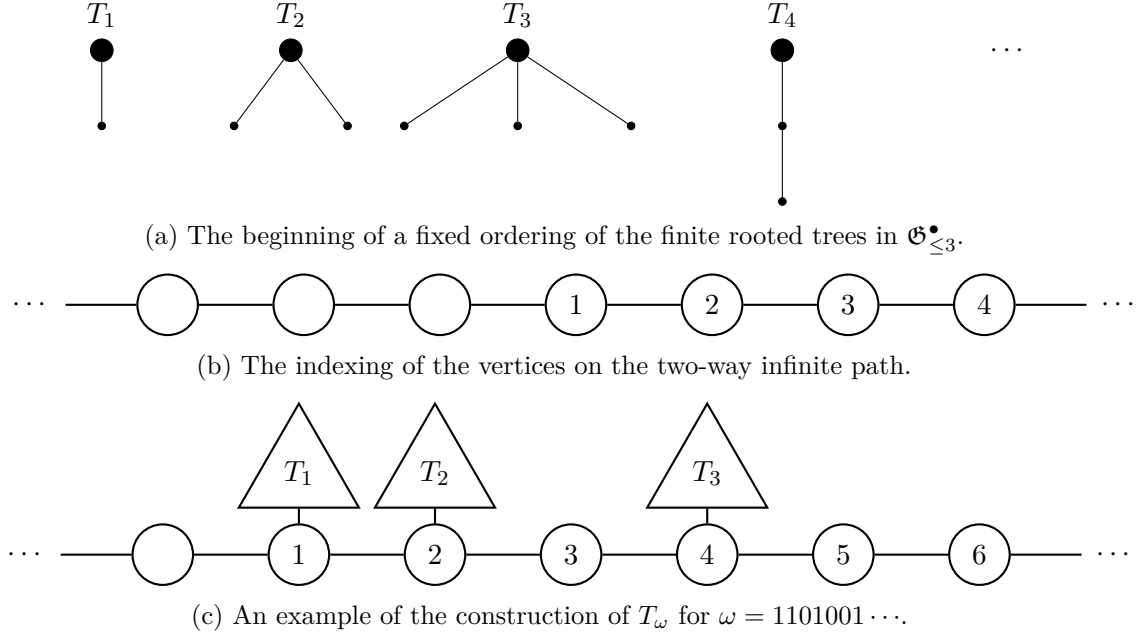
Begin with a two-way infinite path. Choose one vertex and label it \(1\);
then label the vertices to its right by \(2,3,4,\ldots\). The vertices to
its left are left unlabeled (Figure~\ref{fig:pathindex}). Let \(n_i\) denote the position of the \(i\)-th occurrence of \(1\) in \(\omega\). For each \(i\ge 1\), attach the
tree \(T_i\) to the vertex indexed by \(n_i\), connecting it through one of its farthest leaves. Since each sequence contains infinitely many \(1\)'s, all the finite rooted
trees are attached in this construction. Hence
\(
\mathcal{T}^D \subseteq \mathfrak{B}(T_\omega).
\)
Moreover, because \(T_\omega\) contains no cycles, it follows that
\(
\mathcal{T}^D = \mathfrak{B}(T_\omega).
\)

Suppose \(T_\omega \cong T_{\omega'}\). The isomorphism must map the two-way
infinite paths to each other, being the unique such paths, and must match the
attached copies of \(T_1\), which determines the vertex labeled \(1\). The
first copy of \(T_2\) to its right then fixes the direction of labelling, so the
isomorphism preserves all labels. Hence the same tree is attached at each
labeled position in \(T_\omega\) and \(T_{\omega'}\), giving \(\omega=\omega'\).
Thus distinct sequences yield distinct trees, and since \(\Omega\) is
uncountable, so is \(\mathfrak{T}\).
\end{proof}

\bridge*
\begin{proof}
\begin{figure}[H]
\centering
\input{figuers/figure4v2}
\caption{Two graphs with the same collection of rooted balls \(\mathfrak{B}\).
The first graph is bridgeless, while the second contains a two-way infinite bridge.}
\label{fig:twoladders}
\end{figure}

For each \(r\ge 0\), let \(\mathcal{A}^{(=D)}_r\) be the collection of all
rooted \(r\)-balls \(F\) that occur in \(D\)-regular connected infinite graphs and such that
\(F^{+}\) is bridgeless, where \(F^{+}\) is obtained from \(F\) by adding a new
vertex \(x_F\) and joining \(x_F\) to each vertex $u \in L^{<D}_r(F)$ by $D - \deg_F(u)$ parallel edges.  Set
\(
        \mathcal{A}^{(=D)}
        :=
        \bigcup_{r\ge 0}\mathcal{A}^{(=D)}_r .
\)
Indeed, if
\(G\in\mathfrak{G}^{\infty}_{=D}\) has no one-way infinite bridge, then, for
every \(r\ge 0\) and every \(v\in V(G)\), the rooted ball \(B_r(G,v)\) belongs
to \(\mathcal{A}^{(=D)}_r\).

Fix an ordering of \(\mathcal{A}^{(=D)}\), as in
Figure~\ref{fig:graphslist}, and write
\(
        G_1,G_2,G_3,\ldots
\)
for the resulting enumeration. If \(G_i\) is an \(r_i\)-ball, write
\(
        L^{<D}(G_i):=L_{r_i}^{<D}(G_i).
\)

Let \(L_1^{D}\) be obtained from the two-ended \(D\)-regular ladder
\(P_{\mathbb{Z}}\,\square\,K_{D-1}\)---the Cartesian product of the two-way
infinite path \(P_{\mathbb{Z}}\) with the complete graph \(K_{D-1}\)---by
subdividing infinitely many consecutive edges along one fixed copy of
\(P_{\mathbb{Z}}\) and coloring the new vertices blue. Each blue vertex has
degree \(2\), while every other vertex has degree \(D\). For \(D\ge 3\),
\(L_1^{D}\) is bridgeless; see Figure~\ref{fig:generalDladder1}.

We now construct \(L_2^{D}\) from \(L_1^{D}\) by the local modification shown in
Figure~\ref{fig:generalDladder2}: we contract layer \(-1\) (that is,
\(\{-1\}\,\square\,K_{D-1}\)) to a single vertex and layer \(-2\) to a single
vertex. Here, contraction is understood to delete loops and merge parallel edges. The edges between these two layers then collapse to a single edge, which
is a two-way infinite bridge; all blue vertices lie on the same side of it. As
in \(L_1^{D}\), each blue vertex has degree \(2\) and every other vertex has degree
\(D\).

We attach the enumerated graphs
\(
        G_1,G_2,G_3,\ldots
\)
to \(L_1^{D}\) in order. For each \(i\), we attach the vertices of
\(L^{<D}(G_i)\) to the next \(S(G_i)\) available blue vertices of \(L_1^{D}\)
(by \emph{available} we mean those whose degree is still less than $D$),
proceeding from left to right. We make sure not to create parallel edges in this process. 
Eventually, every blue vertex reaches
degree $D$. Let \(H_1^{D}\) be the resulting graph. Then
\(H_1^{D}\) is \(D\)-regular. Moreover, \(H_1^{D}\) is bridgeless: \(L_1^{D}\) is
bridgeless, and each \(G_i^{+}\) is bridgeless, so attaching the deficient
boundary vertices of \(G_i\) to the ladder creates no bridge; see
Figure~\ref{fig:connectladder}.

We construct \(H_2^{D}\) in the same way, using \(L_2^{D}\) instead of \(L_1^{D}\). Again,
\(H_2^{D}\) is \(D\)-regular. The same argument excludes every bridge in \(H_2^{D}\), 
other than the distinguished edge of \(L_2^{D}\),  which is a two-way infinite bridge.
Hence \(H_2^{D}\) has a two-way infinite bridge and no one-way
infinite bridge.

Each \(G_i\in\mathcal{A}^{(=D)}\) appears as a rooted ball in both \(H_1^{D}\)
and \(H_2^{D}\). Indeed, if \(G_i\) is an \(r_i\)-ball, then its \(r_i\)-ball
around the root is unchanged by the attachments, since the new edges are added
only at vertices of \(L_{r_i}(G_i)\). Hence
\(
        \mathcal{A}^{(=D)}
        \subseteq
        \mathfrak{B}(H_1^{D}),\mathfrak{B}(H_2^{D}).
\)

Conversely, \(H_1^{D}\) is bridgeless and \(H_2^{D}\) has no one-way infinite bridge.
Therefore, every finite rooted ball of either graph belongs to
\(\mathcal{A}^{(=D)}\), and so
\(
        \mathfrak{B}(H_1^{D}),\mathfrak{B}(H_2^{D})
        \subseteq
        \mathcal{A}^{(=D)}.
\)
Consequently,
\(
        \mathfrak{B}(H_1^{D})
        =
        \mathfrak{B}(H_2^{D}).
\)

Thus \(d(H_1^{D},H_2^{D})=0\), while \(H_1^{D}\) is bridgeless and \(H_2^{D}\) contains a two-way infinite bridge.
\end{proof}

\regularnonisomorphic*
\begin{proof}
In the construction of \(L_2^{D}\) in the proof of
Theorem~\ref{bridge}, we contracted each of the layers indexed by \(-1\)
and \(-2\) to a single vertex, where contraction is understood to delete
loops and merge parallel edges. More generally, for each \(k\geq 1\), we
may independently choose whether to perform the same contraction on the
pair of layers \(\{-(3k+1),-(3k+2)\}\).
We let $K$ be the set of such $k$, where we do the contraction; 
we only perform the construction for infinite sets~$K$.
For every such choice, after attaching the enumerated graphs as in the
proof of Theorem~\ref{bridge}, the resulting graph remains
\(D\)-regular and has the same collection of rooted balls as \(H_2^{D}\).

Since there are \(2^{\aleph_0}\) possible sets~$K$, this produces
\(2^{\aleph_0}\) graphs with the same collection of rooted balls. It
remains to show that distinct choice sets yield non-isomorphic graphs.
We may identify the original bridge $\{-1,-2\}$ in any isomorphic copy: it is the only bridge 
such that its removal creates one bridgeless component. 
This allows us to distinguish the ladder.  
Every isomorphism must map the original bridge to the original bridge and
preserve its two sides. In particular, it preserves the ordering of the
remaining two-way infinite bridges along the other side of the original
bridge.
The positions of these bridges relative to the original bridge determine
exactly which pairs \(\{-(3k+1),-(3k+2)\}\) were contracted. Hence, if two
resulting graphs are isomorphic, their choice sets~$K$ must coincide.
Therefore, distinct choice sets yield pairwise non-isomorphic graphs.
\end{proof}
\isolateisnotmeager*
\begin{proof}
Let $S$ be a set containing a strongly isolated point $G$, and suppose, for
contradiction, that $S=\bigcup_{n} A_n$ is a countable union of nowhere dense
sets. Then $G\in A_n$ for some $n$. Since $G$ is strongly isolated, the set
$\{G\}$ is open and nonempty, and its only nonempty open subset is $\{G\}$
itself, which meets $A_n$. Hence $A_n$ is not nowhere dense, a contradiction.
\end{proof}

\IsolatedDirected*
\begin{proof}
Infinite outgoing stars (Figure~\ref{fig:star}) are strongly isolated, since
$\vec{S}_{m,\infty}$ is the only graph with the following $1$-balls:
a directed outgoing path of length $2$ with one incoming and one outgoing edge at
the root ($\vec{P_2^\bullet}$), and a directed star with the center as root
($\vec{S}_{m,1}^\bullet$). That is,
\[
\mathfrak{B}_1(\vec{S}_{m,\infty}) = \{\vec{P_2^\bullet},\, \vec{S}_{m,1}^\bullet\}.
\]
Now suppose $\mathfrak{B}_1(G) = \{\vec{P_2^\bullet},\, \vec{S}_{m,1}^\bullet\}$. Since $G$ is connected, starting from the vertex $v$ with
$B_1(G,v) = \vec{S}_{m,1}^\bullet$, every neighbor $w$ of $v$ has indegree at
least one, and hence $B_1(G,w) = \vec{P_2^\bullet}$. The same applies to all
subsequent neighbors, and by induction $G \cong \vec{S}_{m,\infty}$.
\end{proof}

\Uniquedecompositionlemma*
\begin{proof}
Let $\vec{G}$ be an oriented graph, and let $H$ be one of the gadgets
$H_0$, $H_1$, $H_2$. Every edge of $\vec{G}\ast H$ lies in at least one
copy of $H$, namely the copy $H_{(x,y)}$ from the construction; it
remains to show that this copy is unique.

Any isomorphism from $H$ to an induced subgraph of $\vec{G}\ast H$ maps
vertices labeled $q$ to vertices labeled $q$, because these are the
only places in the graph where a path with a triangle is attached (in $H_0$) or a $C_5$
is glued (in $H_1$, $H_2$). Similarly, it maps vertices labeled $p$ to
vertices labeled $p$, because these are the only places with a glued
$C_4$. Since the isomorphism preserves the labels $p$ and $q$, each
vertex of the path containing them is uniquely determined; in
particular, the isomorphism maps $v$ to a copy of $v$ and $w$ to a copy
of $w$. 
The two marker cycles in \(H_0\) are joined by a path of length \(5\), and those in \(H_1,H_2\) by a path of length \(3\), 
with interiors disjoint from the cycles. No such path of the respective length joins the corresponding marker cycles in distinct canonical gadgets.
Hence, for $H_0$ and $H_1$, all vertices of the copy are
uniquely determined. For $H_2$, since $\vec{G}$ is an orientation of a simple graph,
each unordered pair $\{v, w\}$ appears in only one copy, so the induced
length~7 path between them is also uniquely determined. Thus every induced copy
of $H$ is one of the copies $H_{(x,y)}$, which proves the unique
decomposition property for $H_0$, $H_1$, $H_2$.

For a gadget $H$, let $\operatorname{sd}(H)$ denote the graph obtained
from $H$ by subdividing every edge, that is, by inserting a new vertex
in the interior of each edge. For each of the gadgets $H_0$, $H_1$,
$H_2$, the subdivided gadget also has the unique decomposition property:
the same argument applies, with the lengths of the distinguished cycles
and paths doubled. Hence every $\operatorname{sd}^k(H)$, $k \ge 0$, has
the unique decomposition property.

Since subdivision does not increase the maximum degree, the maximum
degree of $\vec{S}_{m,\infty}\ast\operatorname{sd}^k(H)$ is independent
of $k$: it equals $m$ for $H \in \{H_0, H_1\}$ and $m \ge 3$, and $2m$
for $H = H_2$ and $m \ge 2$. Moreover, by the unique decomposition
property, any isomorphism between two such graphs maps gadget copies to
gadget copies; since iterated subdivision strictly increases the
distance between the vertices labeled $p$ and $q$, the graphs
$\vec{S}_{m,\infty}\ast\operatorname{sd}^k(H)$, $k \ge 0$, are pairwise
non-isomorphic.
\end{proof}

\GadgetDecodingLemma*
\begin{proof}
Write $\Delta:=\operatorname{diam}(H)\ge 1$. 
Since every automorphism of $H$ fixes both terminals, each induced copy of $H$ has a determined ordered pair of terminals, 
hence a well-defined orientation.

\medskip
\noindent\textbf{Reduction to a local statement.}
It suffices to prove that, for every $s\in\mathbb{N}$,
\begin{equation}\label{eq:GadgetLocal}
d\!\left(X_n,\,\vec{G}\ast H\right)<2^{-(s+1)\Delta}
\;\Longrightarrow\;
X_n=\vec{G}_n\ast H\ \text{for a unique }\vec{G}_n,
\ \text{and}\ d\!\left(\vec{G}_n,\vec{G}\right)\le 2^{-s}.
\end{equation}
Indeed, assume \eqref{eq:GadgetLocal}. Since $X_n\to\vec{G}\ast H$, there is
$N\in\mathbb{N}$ with $d(X_n,\vec{G}\ast H)<2^{-2\Delta}$ for all $n\ge N$;
applying \eqref{eq:GadgetLocal} with $s=1$, each such $X_n$ has a unique
decomposition $X_n=\vec{G}_n\ast H$, so the skeleton $\vec{G}_n$ is determined by
$X_n$ alone. Moreover, for every $s\in\mathbb{N}$, convergence gives
$d(X_n,\vec{G}\ast H)<2^{-(s+1)\Delta}$ for all sufficiently large $n$, whence
$d(\vec{G}_n,\vec{G})\le 2^{-s}$ eventually. Therefore $\vec{G}_n\to\vec{G}$, as
required.

It remains to prove \eqref{eq:GadgetLocal}. Fix $s$, set $r:=(s+1)\Delta$, and
suppose $d(X_n,\vec{G}\ast H)<2^{-r}$. By the definition of $d$, this gives
\[
\mathfrak{B}_{r}(X_n)\cong\mathfrak{B}_{r}(\vec{G}\ast H),
\]
and, by truncating each rooted ball, $\mathfrak{B}_{k}(X_n)\cong\mathfrak{B}_{k}(\vec{G}\ast H)$
for every $k\le r$; in particular for $k=\Delta$.

\medskip
\noindent\textbf{Existence and uniqueness of the decomposition.}
We note two facts visible inside a ball of radius $\Delta$.
\begin{enumerate}
    \item[\textnormal{(a)}]
    If a vertex $u$ lies in an induced copy $K\cong H$, then any two vertices of
    $K$ are at distance at most $\Delta$, so $K$ is contained in the $\Delta$-ball
    around $u$. Hence whether a given edge incident to $u$ lies in such a copy,
    and in how many, is determined by that ball.
    \item[\textnormal{(b)}]
    If two induced copies of $H$ share a vertex $z$, both are contained in the
    $\Delta$-ball around $z$, so whether they meet only in a common terminal is
    likewise determined by that ball.
\end{enumerate}
Because $H$ has the unique decomposition property, in $\vec{G}\ast H$ every edge
lies in exactly one induced copy of $H$, 
and any two distinct copies intersect in at most one vertex, which is a terminal in both copies.
By (a)--(b) these are properties of $\Delta$-balls, and
$\mathfrak{B}_\Delta(X_n)\cong\mathfrak{B}_\Delta(\vec{G}\ast H)$; hence they hold
in $X_n$ as well. Thus the induced copies of $H$ in $X_n$ partition $E(X_n)$, and
distinct copies meet only at shared terminals.

Consequently $X_n$ is obtained by gluing disjoint copies of $H$ along terminals.
Let $\vec{G}_n$ be the oriented graph whose vertices are the terminal vertices of
$X_n$, with one arc $(x,y)$ for each induced copy of $H$ having ordered terminals
$(x,y)$; the orientation is well defined because $H$ is asymmetric. Replacing each
arc by the corresponding copy recovers $X_n$, so $X_n=\vec{G}_n\ast H$. This
decomposition is unique: any decomposition must use exactly the induced copies of
$H$, which are intrinsic to $X_n$, and the order of their terminals is forced by
the asymmetry of $H$.

\medskip
\noindent\textbf{Closeness of the skeletons.}
A vertex of a graph of the form $\vec{G}\ast H$ is a terminal (equivalently, a
skeleton vertex) precisely when its $\Delta$-ball has the corresponding local
shape, which is detectable within radius $\Delta\le r$. Hence the isomorphism
$\mathfrak{B}_{r}(X_n)\cong\mathfrak{B}_{r}(\vec{G}\ast H)$ restricts to the rooted
$r$-balls based at skeleton vertices.

Finally, in any graph $\vec{G}\ast H$ and for any skeleton vertex $x$, the rooted
skeleton ball $B_s(\vec{G},x)$ is determined by the rooted ball
$B_{(s+1)\Delta}(\vec{G}\ast H,x)$: a skeleton vertex at skeleton-distance at most
$s$ from $x$ lies within distance $s\Delta$ in $\vec{G}\ast H$, and every copy of
$H$ incident to it extends a further $\Delta$ at most, hence lies inside the ball
of radius $(s+1)\Delta$. Collapsing each such copy to its arc therefore
reconstructs $B_s(\vec{G},x)$ exactly, with no arc missed or added. Applying this
to $X_n=\vec{G}_n\ast H$ and to $\vec{G}\ast H$, and using the
skeleton-restricted isomorphism at radius $r=(s+1)\Delta$, we obtain
\[
\mathfrak{B}_s(\vec{G}_n)\cong\mathfrak{B}_s(\vec{G}),
\]
and therefore $d(\vec{G}_n,\vec{G})\le 2^{-s}$. This establishes
\eqref{eq:GadgetLocal} and completes the proof.
\end{proof}
\isolatedproduct*

\begin{proof}
Suppose not. Then, for every \(n\in\mathbb{N}\), there exists
\(X_n\in\mathfrak{G}_{\leq D}\) such that
\(X_n\not\cong\vec{S}_{m,\infty}\ast H\) and
\(
d\bigl(X_n,\vec{S}_{m,\infty}\ast H\bigr)\leq 2^{-n}.
\)
Thus,
\(X_n\to\vec{S}_{m,\infty}\ast H\). By
Lemma~\ref{GadgetDecodingLemma}, for all sufficiently large \(n\), we have
\(X_n=\vec{G}_n\ast H\), where
\(\vec{G}_n\to\vec{S}_{m,\infty}\). By
Lemma~\ref{IsolatedDirected},
\(\vec{G}_n\cong\vec{S}_{m,\infty}\) for all sufficiently large \(n\).
Hence,
\(X_n\cong\vec{S}_{m,\infty}\ast H\), a contradiction.
\end{proof}

\stronglyisolated*
\begin{proof}
By Lemma~\ref{lem:unique-decomposition}, for every \(k\geq 0\), the gadget
\(\operatorname{sd}^k(H_0)\) satisfies the hypotheses of
Corollary~\ref{cor:isolatedproduct}. Moreover, the graphs
\(\vec{S}_{3,\infty}\ast\operatorname{sd}^k(H_0)\), \(k\geq 0\), are
pairwise non-isomorphic and have maximum degree \(3\). Hence, for every
\(D\geq 3\), they form an infinite family of pairwise non-isomorphic
strongly isolated infinite graphs in \(\mathfrak{G}_{\leq D}\).

It remains to show that there are at most countably many strongly isolated
graphs. Let $G$ be a strongly isolated graph. Then there exists $r\in\mathbb N$
such that the open neighborhood $\mathcal{U}_{2^{-(r-1)}}(G)$ contains no graph other than
$G$. Equivalently, $G$ is the unique element of $\mathfrak G_{\mathfrak F}$,
where $\mathfrak F=\mathfrak B_r(G)$.

For each fixed $r$, the degree bound $D$ implies that there are only finitely
many $r$-balls up to isomorphism. Hence, there are only finitely many possible
sets $\mathfrak B_r(G)$. Therefore, only finitely many strongly isolated graphs
can be isolated at scale $r$.

Since every strongly isolated graph is isolated at some scale $r\in\mathbb N$,
the set of strongly isolated graphs is a countable union of finite sets and is
therefore countable. Together with the infinite family constructed above, this
shows that there are countably infinitely many strongly isolated graphs.
\end{proof}

\notmeager*
\begin{proof}
By Lemma~\ref{lem:unique-decomposition}, the gadgets \(H_0,H_1,H_2\),
depicted in Figure~\ref{fig:gadget}, satisfy the hypotheses of
Corollary~\ref{cor:isolatedproduct}. Therefore,
\(\vec{S}_{3,\infty}\ast H_0\) and
\(\vec{S}_{3,\infty}\ast H_1\) are strongly isolated in
\(\mathfrak{G}_{\leq D}\) for every \(D\geq 3\), and \(\vec{S}_{2,\infty}\ast H_2\) is strongly isolated in
\(\mathfrak{G}_{\leq D}\) for every \(D\geq 4\). Since these graphs are
infinite, they are also strongly isolated in the corresponding subspaces
\(\mathfrak{G}_{\leq D}^{\infty}\). We may therefore apply
Lemma~\ref{lem:isolateisnotmeager}.

\medskip
\noindent\textnormal{(i)}
The graph \(\vec{S}_{3,\infty}\ast H_0\) has a one-way infinite bridge.
Hence, for every \(D\geq 3\), the set of infinite graphs with a one-way
infinite bridge is not meager in \(\mathfrak{G}_{\leq D}^{\infty}\).

The graph \(\vec{S}_{3,\infty}\ast H_1\) has no one-way infinite bridge.
Hence, for every \(D\geq 3\), the set of infinite graphs without a one-way
infinite bridge is not meager in \(\mathfrak{G}_{\leq D}^{\infty}\).

\medskip
\noindent\textnormal{(ii)}
Moreover, \(\vec{S}_{3,\infty}\ast H_1\) has a two-way infinite bridge.
Hence, for every \(D\geq 3\), the set of infinite graphs with a two-way
infinite bridge but no one-way infinite bridge is not meager in
\(\mathfrak{G}_{\leq D}^{\infty}\).

For the complementary claim, we use graph $\vec{S}_{1,\infty}\ast H_1$.

\medskip
\noindent\textnormal{(iii)}
The graph \(\vec{S}_{2,\infty}\ast H_2\) is bridgeless. Hence, for every
\(D\geq 4\), the set of infinite bridgeless graphs is not meager in
\(\mathfrak{G}_{\leq D}^{\infty}\).
For the complementary claim, we use graph $\vec{S}_{3,\infty}\ast H_1$. 
\end{proof}

\unrootededgecoloring*
\begin{proof}

\begin{figure}[H]
    \centering

\begin{tikzpicture}[
  scale=.9,
  every node/.style={circle, fill=black, inner sep=1.5pt}
]
  \foreach \i in {1,...,8} {
    \node (v\i) at (\i-1,0) {};
  }
  \draw (v1)--(v2)--(v3)--(v4);
  \draw[densely dotted] (v4)--(v5);
  \draw (v5)--(v6)--(v7)--(v8);

  \node (a) at (1,1) {};
  \node (b) at (2,1) {};
  \node (A) at (5,.8) {};
  \node (B) at (5.5,1) {};
  \node[fill=none, draw] (O) at (4,.8) {$J_D$};
  \node (C) at (6,.8) {};

  \draw (v2)--(a)--(b)--(v3);
  \draw (C)--(B)--(A)--(C);
  \draw (A)--(O);
  \draw (C)--(v7);

  \node[above=3pt, fill=none, inner sep=0pt] at (a) {$p_1$};As
  \node[above=3pt, fill=none, inner sep=0pt] at (b) {$p_2$};
  \node[above=3pt, fill=none, inner sep=0pt] at (A) {$q_1$};
  \node[above=3pt, fill=none, inner sep=0pt] at (C) {$q_2$};
  \node[below=3pt, fill=none, inner sep=0pt] at (v1) {$v$};
  \node[below=3pt, fill=none, inner sep=0pt] at (v8) {$w$};

  \draw[
    decorate,
    decoration={brace, mirror, amplitude=5pt}
  ] (2,-.55)--(6,-.55)
    node[midway, below=8pt, rectangle, fill=none, inner sep=0pt]
    {Length $\ell > |V(J_D)|$};
\end{tikzpicture}

\caption{The gadget $C^D$, which has the unique decomposition property and is not $D$-edge-colorable.}
\label{fig:edge1}
\end{figure}
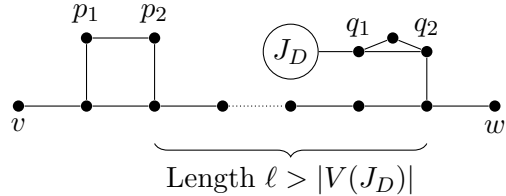
We first show that \(C^D\) (Figure~\ref{fig:edge1}) has the unique decomposition property. 
In any copy of $C^D$ in a graph $\vec{G}\ast C^D$ we can identify the copy of $J_D$
(for $D > 3$ by degrees, for $D = 3$ by the unique diamond, $K_4 - e$; although we need to choose 
edge $uv = ab$ to remove in the construction of $J_D$). After removing this, we use 
Lemma~\ref{lem:unique-decomposition} that says $H_0$ as the unique decomposition property. 
Thus, by Corollary~\ref{cor:isolatedproduct},
both \(\vec{S}_{3,\infty}\ast C^D\) and \(\vec{S}_{3,\infty}\ast H_0\) are
strongly isolated graphs with maximum degree at most \(D\), for every \(D\geq3\).

Since \(H_0\) is $3$-edge colorable and $\vec{S}_{3}$ does not have any cycle, \(\vec{S}_{3,\infty}\ast H_0\) is \(3\)-edge-colorable
(we use one 3-edge-coloring of $H_0$ and permute the colors in each copy, going from the center outwards). 
On the other hand, \(C^D\) contains the obstruction \(J_D\) constructed in the proof of
Theorem~\ref{edgecoloring}. Hence, \(\vec{S}_{3,\infty}\ast C^D\) is not
\(D\)-edge-colorable and, by Vizing's theorem and compactness, has chromatic
index \(D+1\). Therefore, by Lemma~\ref{lem:isolateisnotmeager}, both the class
of \(D\)-edge-colorable graphs and its complement are nonmeager in
\(\mathfrak{G}_{\leq D}^{\infty}\). Consequently, neither class is meager nor
comeager.
\end{proof}

\section{Concluding remarks and open problems}
We have seen that the Baire category framework does not classify as typical or atypical properties of 
unrooted graphs in $\mathfrak{G}_{\le D}^\infty$: 
one-way and two-way infinite bridges for $D\ge3$, bridgelessness for $D\ge4$
as typical or atypical in the unrooted case: each of these classes (and also each of their complements) contains a strongly
isolated graph, and hence none of them is meager.
The strongly isolated graphs constructed in this paper are not regular and have small cuts. This suggests the following questions.

\begin{question}
\label{q:regularisolated}
  For each $k,D\ge3$, does
  $\mathfrak{G}_{=D}^{\infty}$ contain any strongly isolated graph?
  More specifically, does it contain a strongly isolated bridgeless graph,
  or one with a two-way infinite bridge and no one-way infinite bridge?
  For odd $D\ge3$, does it also contain a strongly isolated graph with a
  one-way infinite bridge?
  Is there a cyclically $k$-edge-connected strongly isolated graph?
\end{question}

\begin{question}
  Does $\mathfrak{G}_{\leq 3}^{\infty}$ contain a strongly isolated bridgeless graph?
  Is there a cyclically $k$-edge-connected subcubic strongly isolated graph?
\end{question}

More important is to decide whether strongly isolated graphs are just a convenient proof technique
or a somewhat pathological part of the space that affects genericity. In particular: 
\begin{question}
  If we remove all strongly isolated graphs from $\mathfrak{G}_{\leq D}^{\infty}$ 
  will the class of graphs with one-/two-way infinite bridge be (co)meager?
\end{question}

\section*{Acknowledgments}
We thank Hector Buffi{\`e}re for the helpful discussion about the earlier version of our results. 
Earlier versions of this research have been presented at the fruitful Dynasnet meetings ``Zámeček'' 
and also at Eurocomb 2025. 
ChatGPT and Claude were used to improve the grammar and presentation of the manuscript and also to 
find some of our references. During final polishing, ChatGPT noticed that our techniques prove 
Corollary~3.8 (which we have not noticed ourselves). 

\printbibliography
\end{document}